\documentclass[10pt]{extarticle}
\usepackage{amsfonts, epsfig, amsmath, amssymb, color,amsthm}
\usepackage{textcomp}
\usepackage{paralist}
\usepackage[normalem]{ulem}

\usepackage[english]{babel}
\newcommand{\BB}{\mathbb{B}}
\newcommand{\CC}{\mathbb{C}}
\newcommand{\DD}{\mathbb{D}}
\newcommand{\EE}{\mathbb{E}}
\newcommand{\FF}{\mathbb{F}}

\newcommand{\II}{\mathbb{I}}

\newcommand{\KK}{\mathbb{K}}

\newcommand{\MM}{\mathbb{M}}
\newcommand{\NN}{\mathbb{N}}

\newcommand{\PP}{\mathbb{P}}
\newcommand{\QQ}{\mathbb{Q}}
\newcommand{\RR}{\mathbb{R}}
\newcommand{\SSS}{\mathbb{S}}

\newcommand{\UU}{\mathbb{U}}
\newcommand{\VV}{\mathbb{V}}
\newcommand{\WW}{\mathbb{W}}
\newcommand{\XX}{\mathbb{X}}
\newcommand{\YY}{\mathbb{Y}}

\newcommand{\aA}{\mathcal{A}}
\newcommand{\bB}{\mathcal{B}}
\newcommand{\cC}{\mathcal{C}}
\newcommand{\dD}{\mathcal{D}}
\newcommand{\eE}{\mathcal{E}}
\newcommand{\fF}{\mathcal{F}}
\newcommand{\gG}{\mathcal{G}}

\newcommand{\lL}{\mathcal{L}}

\newcommand{\pP}{\mathcal{P}}

\newcommand{\rR}{\mathcal{R}}
\newcommand{\sS}{\mathcal{S}}

\newcommand{\uU}{\mathcal{U}}

\newcommand{\xX}{\mathcal{X}}
\newcommand{\yY}{\mathcal{Y}}
\newcommand{\zZ}{\mathcal{Z}}

\newcommand{\no}{\noindent}

\newcommand{\ra}{\rightarrow}

\newcommand{\ds}{\displaystyle}

\newtheorem{thm}{Theorem}[section]

\theoremstyle{plain} 
\newtheorem{theorem}{Theorem}[section]
 
\newtheorem{lemma}{Lemma}[section] 
\newtheorem{proposition}{Proposition}[section] 

\theoremstyle{definition} 
\newtheorem{definition}{Definition}[section]

\newtheorem{condition}{Hypothesis}

\theoremstyle{remark} 
\newtheorem{remark}{Remark}[section]
\theoremstyle{definition}

\newtheorem{rem}[thm]{Remark}

\DeclareMathSymbol{\ophi}{\mathalpha}{letters}{"1E}

\renewcommand{\phi}{\varphi}

\newcommand{\be}{\begin{equation}}
\newcommand{\ee}{\end{equation}}
\newcommand{\ben}{\begin{equation*}}
\newcommand{\een}{\end{equation*}}

\newcommand{\ba}{\begin{equation}\begin{aligned}}
\newcommand{\ea}{\end{aligned}\end{equation}}

\newfont{\cyrfnt}{wncyr10}
\def\J3{\cyrfnt{\rm \u{\cyrfnt I}}}
\def\j3{\cyrfnt{\rm \u{\cyrfnt i}}}

\usepackage[]{color}
\definecolor{DarkGreen}{rgb}{0.1,0.7,0.3}   
\newcommand{\cb}[1]{\textcolor{blue}{#1}}

\definecolor{DarkGreen}{rgb}{0.1,0.7,0.3}   

\allowdisplaybreaks[4]

\begin{document}
\title{Moderate averaged deviations for a multi-scale system with jumps and memory
}

\date{\null}

\author{
Pedro Catuogno \footnote{Departamento de Matem\'{a}tica Universidade Estadual de Campinas 13081-970 Campinas SP-Brazil; 
pedrojc@unicamp.br
} \quad \quad \quad \quad \quad 
Andr\'e de Oliveira Gomes \footnote{Departamento de Matem\'{a}tica Universidade Estadual de Campinas 13081-970 Campinas SP-Brazil; ENSTA-ParisTech Applied Mathematics Department,
828 Boulevard des Maréchaux, 91120 Palaiseau, France; andre.deoliveiragomes@cardis.io}  \hspace{2cm}
}

\maketitle



\begin{abstract}
This work studies a two-time-scale functional system given by two jump-diffusions under the scale separation by a small parameter $\varepsilon \ra 0$. The coefficients of the equations that govern the dynamics of the system depend on the segment process of the slow variable (responsible for capturing delay effects on the slow component) and on the state of the fast variable. We derive a moderate deviations principle for the slow component of the system in the small noise limit using the weak convergence approach. The rate function is written in terms of the averaged dynamics associated to the multi-scale system. The core of the proof of the moderate deviations principle is the establishment of an averaging principle for the auxiliary controlled processes associated to the slow variable in the framework of the weak convergence approach. The controlled version of the averaging principle for the jump multi-scale diffusion relies on some discretization method inspired by the classical Khasminkii's averaging principle.
\end{abstract}

 \section{Introduction}
  \no Fixed a terminal time $T>0$ and a certain delay $\tau>0$ we consider in the small noise limit $\varepsilon \ra 0$ the two-time scale stochastic system given for any $t \in [0,T]$ by 
\begin{align} \label{eq: the multiscale system int} 
\begin{cases}
dX^\varepsilon(t)&= a(X^\varepsilon_t,Y^\varepsilon(t))dt + \sqrt{\varepsilon} \sigma(X^\varepsilon_t)dB^1(t) + \displaystyle \int_{\XX }\cb{\varepsilon} c(X^\varepsilon_{t-},z) \tilde N^{\frac{1}{\varepsilon}}(t,dz);  \\
d Y^\varepsilon(t)&= \displaystyle \frac{1}{\varepsilon} f(X^\varepsilon_t, Y^\varepsilon(t))dt + \frac{1}{\sqrt{\varepsilon}} g(X^\varepsilon_t, Y^\varepsilon(t)) dB^2(t) + \displaystyle \int_\XX h(X^\varepsilon_{t-}, Y^\varepsilon(t-),z) \tilde N^{\frac{1}{\varepsilon}}(t,dz).
\end{cases}
\end{align}

\no For every $\varepsilon>0$ the stochastic process $(X^\varepsilon(t), Y^\varepsilon(t))_{t \in [0,T]}$ takes values in $\RR^n:=\RR^d \times \RR^k$. The initial datum is $(X^\varepsilon_0, Y^\varepsilon(0))=(\chi,y)$ where $\chi$ is a given continuous function from $[-\tau,0]$ to $\RR^d$ (initial delay segment) and $y \in \RR^k$.  The processes $X^\varepsilon$ and $Y^\varepsilon$ are denominated the slow variable and respectively the fast variable of the multi-scale stochastic system (\ref{eq: the multiscale system int}). We stress that we use the notation $X^\varepsilon_t$ for the segment process, i.e. $X^\varepsilon_t := \{ X^\varepsilon(t+\theta) ~|~\theta \in [-\tau,0]\}$ for any $t \geq 0$. We refer the reader to Chapters 5 and 6 of the book \cite{Mao book} for an introduction to the subject of stochastic functional differential equations with Brownian noise and to \cite{di Nunno et al} for the study of stochastic functional differential equations with jumps. The space of the jump increments $\XX$ is Euclidean, the process $B=(B^1,B^2)$ is a standard Brownian motion (BM for short) with values in $\RR^n$ with first component $B^1$ a standard BM with values in $\RR^d$ and second component $B^2$ an independent $\RR^k$-valued standard BM. For every $\varepsilon>0$ the random measure $\tilde N^{\frac{1}{\varepsilon}}$ is an independent compensated Poisson random measure with intensity given by $ds \otimes \frac{1}{\varepsilon} \nu(dz)$, where $ds$ stands for the Lebesgue measure on the real line and $\nu$ is a L\'{e}vy measure on $\XX$. In this work we consider $\nu$ possibly with infinite total mass but satisfying an exponential integrability condition that reads as the big jumps of the underlying L\'{e}vy process having exponential moments of order 2. The assumptions on the coefficients of (\ref{eq: the multiscale system int}) and on the measure $\nu$ will be precised with full rigour in the following section. \\

 \no Multi-scale stochastic systems as (\ref{eq: the multiscale system int}) are nowadays very popular in applied mathematical and physical disciplines since they are sucessful models for phenomena exhibiting different levels of heterogeneity/homogeneity that can be asymptotically categorized by scaling. This technique of understanding diversity exploits the decomposition of the phase space of the model in two sets of variables, the ones with slow degrees of freedom and the ones with fast degrees of freedom through a separation scale given by an intensity parameter measuring this degree of heterogeneity/homogeneity. We refer the reader to \cite{Liu and Vanden-Eijden} and the monograph \cite{Pavliotis Stuart} for an introduction to the subject. Typical examples are multi-factor stochastic volatility models in Finance \cite{Papanicolaou1, Papanicolaou2} and the dynamics of proxy-data in Climatology \cite{Kifer} where climatic transitions are understood within the distinction between slow and fast variables that encode different factors used to build statistical parametrizations. In the description of those climatic models short/large time-scales must be taken into consideration (e.g. daily weather forecast vs climatic prediction) 
 in order to see interesting phenomena such as metastability of the slow variable from an equilibrium state of the deterministic dynamics (cf. Appendix in \cite{Debussche}). Often in these multi-scale climatic models the slow variable quantifies data related with large time scales (e.g. climatic data). Multi-scale stochastic 
 systems of the type (\ref{eq: the multiscale system int}) offer the mathematical formalism necessary to capture more realistic attributes of the underlying stochastic climate model.  The paradigmatic example in climate dynamics is the coupling of ocean temperatures models (slow variable) with the atmospheric Lorentz equations (fast variable). We refer the reader for more details to \cite{Dij13}. 
  The presence of an underlying L\'{e}vy process that drives the stochastic dynamics of (\ref{eq: the multiscale system int}) in small noise models abrupt climate transitions. A typical example is given by the \textit{Daansgard-Oeschger events} that show statistical evidence of underlying jump noise signals (cf. Chapter 10 in \cite{Dij13} and \cite{Ditlevsen, Gairing et al.16, Hein et al.09}). The dependence of the coefficients of (\ref{eq: the multiscale system int}) on the the segment process of the slow variable models the memory effects exhibited by energy balance models such as the ones constucted in \cite{Dij19}. \\

\no This type of multi-scale systems are highly complex and difficult to analyse or simulate. It is desirable to approximate in a suitable sense the dynamics of the slow variable by some simpler dynamical system. The idea of the averaging principle performed first by Khasminkii in \cite{Khasminskii} is the following. Under strong dissipativity assumptions on the coefficients of the fast variable that ensure the existence of a unique invariant measure $\mu^{\zeta}$ for the fast variable process with frozen slow variable $\zeta$ and such that a certain ergodic property holds for the mixing coefficient $a$ wrt to its average against $\mu^\zeta$ (cf. Proposition \ref{proposition: the averaging property for the averaged coefficient}) 
\begin{align} \label{eq: int-the averaged coef}
\bar a(\zeta):= \int_{\RR^k} a(\zeta,y) \mu^{\zeta}(dy)
\end{align} 
the (strong) averaging principle states that for any $T>0$ and $\delta>0$ one has
\begin{align} \label{eq: int limit1}
\displaystyle \lim_{\varepsilon \ra 0} \PP \Big ( \displaystyle \sup_{t \in [0,T]} |X^\varepsilon(t)- \bar X(t)|> \delta \Big )=0,
\end{align}
where $\bar X^0$ is the unique solution of the functional averaged differential equation  
\begin{align} \label{eq: int the averaged dynamics}
\begin{cases}
\frac{d}{dt} \bar X^0(t)&= \bar a(\bar X^0_t), \quad t \in [0,T];\\
\bar X^0_0&=\chi.
\end{cases}
\end{align}
\no The averaging principle has applications to problems in celestial (stochastic) mechanics (cf. Chapter 7 in \cite{FW98}) and climatic energy balanced models (cf. \cite{Arnold}) among others and has a rich and diverse history in the literature. Khasminkii's technique was introduced in \cite{Khasminskii} and later implemented by Mark Freidlin \cite{Freidlin78} and Veterennikov in \cite{Vetennikov} in different contexts, finding huge applicability in a diverse range of problems. We refer the reader to the following exemplary but not exhaustive works on weak and strong averaging principles: \cite{Cerrai, Cerrai2, Cerrai4} concerning multi-scale systems constituted by stochastic partial differential equations (SPDEs for short) driven by space time white noise; \cite{Givon, Liu, Xu-Miao-Liu, Xu} for multi-scale (finite and infinite dimensional) systems constituted by jump-diffusions and \cite{Bao et al, Mao et al} for stochastic dynamical systems with coefficients functionally dependent with delay. Although the averaging principle (\ref{eq: int limit1}) yields an approximation result for small $\varepsilon>0$ of the slow variable process by the averaged dynamics of $\bar X$ nothing is said on the rate of convergence. Large and moderate deviations type of statementes provide sharper estimates within the identification of a rate of convergence for the limit (\ref{eq: int limit1}) in an exponentially small scale in $\varepsilon \ra 0$ and in terms of a deterministic quantity designated good rate function. We refer the reader to \cite{BDG18, Duan12, Kumar17, Vetennikov2}, for stochastic averaging under the large deviations regime and respectively to  \cite{Feng-Fouque-Kumar, Guillin00, Guillin03} for averaging under moderate deviations regimes. \\
 
 \no The aim of this article is to derive a moderate deviations principle (MDP for short) for $(X^\varepsilon)_{\varepsilon>0}$ as $\varepsilon \ra 0$. More precisely we will study deviations of $X^\varepsilon$ from the averaged dynamical system $\bar X$, that is
\begin{align*}
Z^\varepsilon:= \frac{X^\varepsilon- \bar X^0}{a(\varepsilon)} \quad \text{ as } \varepsilon \ra 0,
\end{align*}
for certain families of magnitude scales $a(\varepsilon)$ such that $a(\varepsilon) \ra 0$ and $b(\varepsilon):=\frac{\varepsilon}{a^2(\varepsilon)} \ra 0$ as $\varepsilon \ra 0$. We fix $\theta \in  \Big (\frac{1}{2}, 1 \Big )$ and let $b(\varepsilon):=\varepsilon^\theta$, $\varepsilon>0$.  The restrictions on the range of $\theta$ are due to parametric choices that are used in the course of the proof. This can be appreciated in the course of the proof of the technical but crucial Lemma \ref{lemma: Khasminkii segment process estimate} in the Appendix. Although we impose the restrictions on the magnitudes $a(\varepsilon)$ as stated above, the free parameter $\theta \in  \Big (\frac{1}{2}, 1 \Big ) $ still covers a big range of moderate deviations intermediary regimes. 
Assuming specific hypotheses on the coefficients that guarantee that $\bar a$ defined in (\ref{eq: int-the averaged coef}) exists, it is Fr\'{e}chet differentiable with Lipschitz derivative and that the L\'{e}vy measure $\nu$ satisfies a certain exponential integrability property, we prove that the family $(Z^\varepsilon)_{\varepsilon>0}$ satisfies a moderate deviations principle with speed $b(\varepsilon) \ra 0$ in $\DD([0,T]; \RR^d)$, the space of c\'{a}dl\'{a}g functions endowed with the Skorokhod topology, and the good rate function $\II: \DD([0,T]; \RR^d) \longrightarrow [0, \infty]$ given by
\begin{align*}
\II(\eta) := \displaystyle \inf_{(f,h) \in L^2([0,T]) \times L^2 (\nu \otimes ds)}  \frac{1}{2} \Big ( \int_0^T |f(s)|^2 ds + \int_0^T |h(s,z)|^2 \nu(dz) ds \Big ),
\end{align*}
where for every $(f,h) \in L^2([0,T]) \times L^2(\nu \otimes ds)$ the function $\eta \in C([-\tau,T]; \RR^d)$ solves uniquely the  skeleton equation: 
\begin{align} \label{eq: the controlled dynamics int}
\begin{cases}
\eta(t)&= \displaystyle \int_0^t D \bar a(\bar X^0_s) \eta_s ds + \int_0^t \cb{\sigma}(\bar X^0_s) f(s)ds + \int_0^t \int_\XX c(\bar X^0_s,z)h(s,z) \nu(dz)ds, \quad t \in [0,T]; \\
\eta_0&=0
\end{cases}
\end{align}
and the function $\bar X^0 \in C([-\tau,T]; \RR^d)$ is the unique solution of (\ref{eq: int the averaged dynamics}). \\

\no This means that the functional $\II$ has compact sublevel sets $\{ \II \leq c \}$ in the Skorokhod topology for any $c\geq 0$ and that for any open set $G \in \bB(\DD([0,T]); \RR^d))$ and closed set $F \in \bB(\DD([0,T]; \RR^d))$ the following holds:
\begin{align*}
\displaystyle \liminf_{\varepsilon \ra 0} \varepsilon^\theta \ln \PP(Z^\varepsilon \in G) &\geq - \displaystyle \inf_{\eta \in G} \II(\eta) \quad \text{and} \\
\displaystyle \limsup_{\varepsilon \ra 0} \varepsilon^\theta \ln \PP(Z^\varepsilon \in F) & \leq - \displaystyle \inf_{\eta \in F} \II(\eta).
\end{align*}

\no We stress that the moderate deviations regime of speed $b(\varepsilon)= \varepsilon^\theta$, $\theta \in \Big (\frac{1}{2},1 \Big ) $, is an intermediary regime between the central limit approximation $a(\varepsilon)= \sqrt{\varepsilon}$ and the large deviations regime $a(\varepsilon)=1$. 
The moderate deviations regime is a very desirable asymptotic regime for the sake of applications since the rate function involves a quadratic functional which is often easier to use in applied problems in comparison with the more involved forms of the rate function used in large deviations statements. We refer as examples \cite{Friz, Jacquier} for the application of moderate deviations principles in Finance, \cite{Guillin2} in Statistics and \cite{Keblaner} where the moderate deviations regime is used to study asymptotics of exit times results for discrete random dynamical systems. \\

\no In order to prove our result we use the weak convergence approach of Dupuis, Ellis, Budhiraja and collaborators that rely on the equivalence in Polish spaces between the definition of large deviations principle and the variational principle nowadays known in the literature as the Laplace-Varadhan principle. Initially, Fleming applied in \cite{Fleming1, Fleming3} methods of stochastic control to large deviations problems. The control-theoretical approach was carried out later in order to derive variational formulas for Laplace functionals of Markov processes in different contexts (cf.\cite{Dupuis Ellis}). In \cite{BD00} the authors derive a sufficient condition for large deviations principles (LDPs for short) for Brownian diffusions and later for jump-diffusions in \cite{BDM11, BCD13} through the establishment of variational formulas for Laplace functionals of Markov processes. We refer the reader for the recent book \cite{Budhiraja book} for a up-to-date 
introduction to the subject. In \cite{BDG15} Budhiraja, Dupuis and Ganguly derive a sufficient condition for a MDP that was successfully applied  in \cite{BW16} and in \cite{Zheng-Zhai-Zhang} to the study of MDPS for SPDEs. The literature on large/moderate deviations principles for stochastic differential equations with delay is not so extensive such as in other domains of applications. We refer the reader to the works \cite{Azencott} and \cite{Lipshutz}  where the authors apply Freidlin-Wentzell types of LDPs to the study of the first exit time problem in the small noise limit for Gaussian diffusions with delay. For the application of the weak convergence approach in the establishment of MDPs to stochastic differential delay equations we mention the works \cite{Ma-Xi, Suo-Tao-Zhang}. 

\paragraph{Strategy of the proof.}

\no The proof of the main result of this work follows from an abstract sufficient condition for moderate deviations principles stated as Theorem 9.9 in \cite{Budhiraja book}. In our case the application of this abstract condition is not straightforward due to the coupling between the slow variable $X^\varepsilon$ and the fast variable $Y^\varepsilon$ in (\ref{eq: the multiscale system int}) with different scaling orders in $\varepsilon \ra 0$.  

\no More precisely the difficult part is to prove directly the following. Fix $\beta \in (0,1)$, $M \geq 0$, two families of random variables $(\xi^\varepsilon)_{\varepsilon>0}$ and $(\psi^\varepsilon:= \frac{\varphi^\varepsilon-1}{a(\varepsilon)})_{\varepsilon>0}$ such that for any $\varepsilon>0$ one has $\int_0^T |\xi^\varepsilon(s)|^2 ds \leq M a^2(\varepsilon)$, where $\varphi^\varepsilon \geq 0$ satisfies $\int_0^T \int_\XX (\varphi^\varepsilon(s,z) \ln \varphi^\varepsilon(s,z) - \varphi^\varepsilon(s,z) +1) \nu(dz) ds \leq M$ $\PP$-a.s. obeying the following convergences in law, $\xi^\varepsilon \Rightarrow \xi$ in the $L^2$-weak topology and $\psi^\varepsilon \textbf{1}_{\{|\psi^\varepsilon| \leq \frac{\beta}{a(\varepsilon)}\}} \Rightarrow \psi$ in some ball of $L^2(\nu \otimes ds)$ equipped with the respective $L^2$-weak topology. Consider the family $\zZ^\varepsilon := \frac{\xX^\varepsilon - \bar X^0}{a(\varepsilon)}$, $\varepsilon>0$, where $(\xX^\varepsilon)_{\varepsilon>0}$ is defined for every $\varepsilon>0$ and $t \in [0,T]$ by
\begin{align} \label{eq: Khasminki- controlled X variable}
\begin{cases}
 \xX^\varepsilon(t) &= \xi(0) + \displaystyle \int_0^t \Big ( a(\xX^\varepsilon_s, \yY^\varepsilon(s)) + \sigma(\xX^\varepsilon_s) \xi^\varepsilon_1(s) + \int_\XX c(\xX^\varepsilon_s,z) (\varphi^\varepsilon(s,z)-1) \nu(dz)\Big ) ds  \\
&+ \sqrt{\varepsilon} \displaystyle \int_0^t \sigma(\xX^\varepsilon_s) dB^1(s) + \varepsilon \int_0^t \int_\XX c(\xX^\varepsilon_{s-},z) \tilde N^{\frac{1}{\varepsilon} \varphi^\varepsilon}(ds,dz); \\
\xX^\varepsilon_0&= \xi.
\end{cases}
\end{align}
and
\begin{align} \label{eq: Khasminkii- controlled Y variable}
\begin{cases}
\yY^\varepsilon(t) &=y + \displaystyle \frac{1}{\varepsilon}\int_0^t \Big (f(\xX^\varepsilon_s, \yY^\varepsilon(s)) + g(\xX^\varepsilon_s, \yY^\varepsilon(s)) \xi^\varepsilon_2(s) + \int_\XX h(\xX^\varepsilon_s, \yY^\varepsilon(s),z) (\varphi^\varepsilon(s,z)-1) \nu(dz) \Big ) ds \\
& + \displaystyle \frac{1}{\sqrt{\varepsilon}} \int_0^t g(\xX^\varepsilon_s, \yY^\varepsilon(s)) dB^2(s) + \int_0^t \int_\XX h(\xX^\varepsilon_{s-}, \yY^\varepsilon(s-),z) \tilde N^{\frac{1}{\varepsilon} \varphi^\varepsilon}(ds,dz);\\
\yY^\varepsilon_0 &=y;
\end{cases}
\end{align}
where for any $\varepsilon>0$ the random measure $\tilde N^{\frac{1}{\varepsilon} \varphi^\varepsilon}$ is a controlled random measure that under a change of probability measure has the same law of $\tilde N^{\frac{1}{\varepsilon}}$ under the original probability measure. This will be rigorously stated in Section \ref{section: proof}. \\

\no Under the following setting, the main task in the derivation of the MDP is to prove that $\zZ^\varepsilon \Rightarrow \bar \zZ$ where $\bar \zZ$ solves (\ref{eq: the controlled dynamics int}) uniquely in $C([-\tau,T]; \RR^d)$ for the control $(f,g)= (\xi, \psi) \in L^2([0,T]) \times L^2(\nu \otimes ds)$. In order to prove that convergence in law we show that the family $(\xX^\varepsilon)_{\varepsilon>0}$ satisfies a tightened averaging principle, i.e. for every $\delta>0$ the following holds
\begin{align} \label{eq: the controlled averaging principle int}
\displaystyle \limsup_{\varepsilon \ra 0} \PP \Big ( \displaystyle \sup_{t \in [0,T]} |\xX^\varepsilon(t) - \bar \xX^\varepsilon(t)|> \delta a(\varepsilon) \Big )=0,
\end{align}
where $(\bar \xX^\varepsilon)_{\varepsilon>0}$ is defined for every $\varepsilon>0$ and $t \in [0,T]$ by
\begin{align} \label{eq: Khasminkii- controlled averaged variable}
\begin{cases}
\bar \xX^\varepsilon(t) &= \chi(0)+ \displaystyle\int_0^t \Big ( \bar a(\bar \xX^\varepsilon_s) + \sigma(\bar \xX^\varepsilon_s) \xi^\varepsilon_1(s) + \int_\XX c(\bar \xX^\varepsilon_s,z) (\varphi^\varepsilon(s,z)-1) \nu(dz) \Big ) ds  \\ 
& + \displaystyle \sqrt{\varepsilon} \int_0^t \sigma(\bar \xX^\varepsilon_s) dB^1(s) + \varepsilon \int_0^t \int_\XX c(\bar \xX^\varepsilon_{s-},z) \tilde N^{\frac{1}{\varepsilon} \varphi^\varepsilon}(ds,dz);\\
\bar \xX^\varepsilon_0&= \xi.
\end{cases}
\end{align}
This will imply by Slutzky's theorem (Theorem 4.1 in \cite{Billingsley}) that $(\zZ^\varepsilon)_{\varepsilon>0}$ has the same weak limit of $(\bar \zZ^\varepsilon)_{\varepsilon>0}$ where $\bar \zZ^\varepsilon:= \frac{\bar \xX^\varepsilon -\bar X^0}{a(\varepsilon)}$, $\varepsilon>0$. And therefore we are conducted to the (easier) task to show that $\bar \zZ^\varepsilon \Rightarrow \bar Z$ (since the dynamics of (\ref{eq: Khasminkii- controlled averaged variable}) is decoupled from the dynamics of the fast variable of the original stochastic system (\ref{eq: the multiscale system int})). \\

\no The proof that $\bar \zZ^\varepsilon \Rightarrow \bar Z$ as $\varepsilon \ra 0$ relies on classical arguments of weak convergence. We use localization techniques in order to obtain good estimates for the second moment of the processes in combination with the Bernstein's inequality for c\`{a}dl\`{a}g local martingales given in the form of Theorem 3.3 of \cite{DZ01} implying the tightness of the respective laws. Hence the relative compactness of the laws follows yielding, due to  Skorohod's representation together with the well-posedness of the skeleton equation (\ref{eq: the controlled dynamics int}), the desired conclusion. \\

\no The proof of the tightened controlled averaging principle (\ref{eq: the controlled averaging principle int}) is inspired on the classical Khasminkii's technique introduced in \cite{Khasminskii}. In a nutshell the procedure relies on a discretization of the time interval $[0,T]$ and the delay initial interval $[-\tau,0]$ in a finite number of intervals with same length $\Delta(\varepsilon) \ra 0$ as $\varepsilon \ra 0$ satisfying some growth conditions that will interplay with the ergodic properties of the averaged dynamics via the construction of auxiliary processes $(\hat \xX^\varepsilon)_{\varepsilon>0}$ and $(\hat \yY^\varepsilon)_{\varepsilon>0}$. The construction of the auxiliary processes is not a straightforward generalization of the Khaminkii's type of discretizations used to prove the usual strong averaging principle. In our setting we need to build stable not-straightforward discretizations $(\hat \xX^\varepsilon)_{\varepsilon>0}$ and $(\hat \yY^\varepsilon)_{\varepsilon>0}$ in order to deal with the nonlocal integral terms that appear in the structure of the respective equations of $(\xX^\varepsilon)_{\varepsilon>0}$ and $(\yY^\varepsilon)_{\varepsilon>0}$. The proof of (\ref{eq: the controlled averaging principle int}) builds heavily on the derivation of stable estimates for the deviations of the segment process $(\hat \xX^\varepsilon_t)_{t \in [0,T]}$ from the slow variable's segment $(\xX^\varepsilon_t)_{t \in [0,T]}$ and respectively the deviations of the approximation $(\hat \yY^\varepsilon(t))_{t \in [0,T]}$ from the fast variable controlled process $(\yY^\varepsilon(t))_{t \in [0,T]}$. We derive asymptotic bounds in $\varepsilon>0$ for the second moment of the deviations of the fast variable from its discretization in contrast with the way we estimate the respective deviations of the slow segment from its approximation. Due to dependence on the segment process given in the dynamics of $(\xX^\varepsilon)_{\varepsilon>0}$ it turns out to be better to control the probability of the slow component deviations for the purpose of obtaining (\ref{eq: the controlled averaging principle int}). This is a technical but major distinction of the technique for obtaining the strong controlled averaging principle (\ref{eq: the controlled averaging principle int}) in comparison with the usual techniques available in the literature.\\

\no Our main result shows in particular that $(X^\varepsilon)_{\varepsilon>0}$ obeys the same moderate deviations principle of $(\bar X^\varepsilon)_{\varepsilon>0}$ where we define the averaged process $\bar X^\varepsilon$ for every $\varepsilon>0$ and $t \in [0,T]$ by
\begin{align*}
\bar X^\varepsilon(t)= \zeta(0) + \int_0^t \bar a( \bar X^\varepsilon_s) ds + \sqrt{\varepsilon} \int_0^t \sigma(\bar X^\varepsilon_s)dB^1(s) + \varepsilon \int_0^t \int_\XX c(\bar X^\varepsilon_{s-},z) \tilde N^{\frac{1}{\varepsilon}}(ds,dz).
\end{align*}
\no One could firstly derive the moderate deviations principle for $(\bar X^\varepsilon)_{\varepsilon>0}$ and secondly show that the families $(X^\varepsilon)_{\varepsilon>0}$ and $(\bar X^\varepsilon)_{\varepsilon>0}$ are exponentially equivalent, i.e. for every $\delta>0$ we have
\begin{align} \label{eq: int-exponential negligibility}
\displaystyle \lim_{\varepsilon \ra 0} \frac{\varepsilon}{a^2(\varepsilon)} \ln \PP \Big ( \sup_{0 \leq t \leq T} \Big | \frac{X^\varepsilon(t) - \bar X^\varepsilon(t)}{a(\varepsilon)} \Big | > \delta\Big )=-\infty.
\end{align}
This would imply that $(X^\varepsilon)_{\varepsilon>0}$ obeys the same MDP of $(\bar X^\varepsilon)_{\varepsilon>0}$ as $\varepsilon \ra 0$. However verifying the exponential equivalence of those families is in general hard. The reasoning employed in this work illustrates the robustness of the weak convergence approach providing a way of reducing the proof of the MDP to the verification of properties concerning continuity and tightness of certain auxiliary processes associated to $(X^\varepsilon)_{\varepsilon>0}$. Such reduction of complexity in such endeavour can be appreciated immediately by the contrast between the $0$ scale of the limit (\ref{eq: the controlled averaging principle int}) with the exponential negligibility demanded in the establishment of the limit (\ref{eq: int-exponential negligibility}).

\paragraph*{Notation.} The arrow $\Rightarrow$ means convergence in distribution. Throughout the article we use when convenient the shorthand notation $A (\varepsilon)\lesssim_\varepsilon B(\varepsilon)$ to mean  that there exist a constant $ c>0$ independent of $\varepsilon>0$ and $\varepsilon_0>0$ such that $A(\varepsilon) \leq cB(\varepsilon)$ for every $\varepsilon< \varepsilon_0$. We write $A(\varepsilon) \simeq_\varepsilon B(\varepsilon)$ as $\varepsilon \ra 0$ to mean that $A(\varepsilon) \lesssim_\varepsilon B(\varepsilon)$ and $B(\varepsilon) \lesssim_\varepsilon A(\varepsilon)$ as $\varepsilon \ra 0$.

\paragraph*{Outline of the paper.} In section \ref{section: statement and preliminaries} we state with full detail the probabilistic framework and the hypothesis on the coefficients of (\ref{eq: the multiscale system int}) in order to state with full rigour the already announced MDP for the family $(Z^\varepsilon)_{\varepsilon>0}$. We finish that section with some examples. Section \ref{section: proof} contains the proof of the main result following the already announced strategy with full detail. The Appendix contains for the reader's convenience technical auxiliary results that can be skipped in a first reading. 
\section{Preliminaries and statement of the main theorem} \label{section: statement and preliminaries}

\subsection{The probabilistic and functional setup. The averaged dynamics.}
\subsubsection{The probabilistic setup and notation.}
\no We follow extensively the probabilistic ansatz and the notation introduced by Budhiraja, Dupuis, Maroulas and collaborators in \cite{BDM11, BCD13, BDG15} and systematized in \cite{Budhiraja book}. For any $\SSS$ topological space we denote by $\bB(\SSS)$ its Borel $\sigma$-algebra. Fix $T>0$, $n=d+k$ with $d,k \in \NN$ and let $\WW= C([0,T];\RR^n)$ endowed with the topology of the uniform convergence  which turns out to be a Polish space. 
Let $\XX= \RR^d \backslash \{0\}$ and $\MM$ be the space of locally finite measures defined on $(\XX, \bB(\XX))$. We endow $\MM$ with the weakest topology such that  for every $f \in \CC_c(\XX)$ (the space of compactly supported continuous functions)  the function $\nu \mapsto \langle \nu , f \rangle:= \int_{\XX} f(u) \nu(du)$, $\nu \in \MM$, is continuous. This topology is known as the vague topology and can be metrized such that $\MM$ turns out to be a Polish space. We refer the reader to \cite{BDM11}.\\
 Fix a measure $\nu \in \MM$ and let $\nu_T = ds \otimes \nu$ where $\ds$ is the Lebesgue measure on $[0,T]$. Consider the product space $\VV= \WW \times \MM$ and denote by $\PP$ the unique probability measure on $(\VV, \bB(\VV))$ under which the first projection $B: \VV \longrightarrow \WW$,  $B(\beta,m)= \beta$ is a standard Brownian motion with values in $\RR^n$ and $N: \VV \longrightarrow \MM$,  $N(\beta,m):= m$ is a Poisson random measure with intensity measure $\nu_T$ The corresponding expectation operator will be denoted by $\EE$. We refer the reader to Theorem I.9.1 in \cite{Ikeda Watanabe}. 
Let $\YY:= \XX \times [0,\infty)$, $\YY_T:= [0,T] \times \YY$, write $\bar \MM$ for the space of the locally finite measures defined on $\YY_T$ when equipped with its Borel $\sigma$-algebra and $\bar \VV:= \WW \times \bar \MM$. In a slight abuse of notation and analogously to what was said to $\MM$, the space $\bar \MM$ turns out to be also a Polish space and there exists a unique probability measure $\cb{\bar \PP}$defined on $(\bar \VV, \bB(\bar \VV))$ such that the maps $B: \bar \VV \longrightarrow \WW$, $B(\beta, \bar m):= \beta$ is a standard Brownian motion with values in $\RR^n$ and $\bar N: \bar \VV \longrightarrow \bar \MM$, $\bar N(\beta,\bar m):= \bar m$ is a Poisson random measure with values on $\bB(\RR^d \times \RR^d \backslash \{0\} \times [0,\infty))$ and intensity measure given by $ds \otimes \nu \otimes dr$, where $dr$ stands for the Lebesgue measure on $([0,\infty); \bB([0,\infty))$.

\no For every $\varepsilon>0$ we consider $N^{\frac{1}{\varepsilon}}$ the Poisson random measure defined on the probability space $(\VV, \bB(\VV))$ with intensity measure given by $\frac{1}{\varepsilon} ds \otimes \nu \otimes dr$ and $\tilde N^{\frac{1}{\varepsilon}}$ for its compensated counterpart.  We also regard when necessary the object $N^{\frac{1}{\varepsilon}}$ as a controlled random measure on $(\bar \VV, \bB(\bar \VV))$ (and therefore $\bB(\bar \VV)$-measurable) under $\bar \PP$ by the identity 
\begin{align} \label{eq: the controlled random measure}
N^{\frac{1}{\varepsilon}}((0,t] \times U) := \int_0^t \int_U \int_0^\infty \textbf{1}_{[0,\frac{1}{\varepsilon}]}(r) \bar N(ds,dx,dr), \quad t \in [0,T], U \in \bB(\XX).
\end{align}
We remark that the space $\YY:= \XX \times [0,\infty)$ takes into account the jumps and the frequencies of the underlying Poisson random measure $N$ and refer the reader to \cite{BDM11} for more details. \\

For any $t \in [0,T]$ define 
\begin{align*}
\fF_t:= \sigma \{ \bar N((0,s] \times A); B(s) ~|~ 0 \leq s \leq t, A \in \bB(\YY) \}
\end{align*}
and denote by $\bar \FF:= \{\bar \fF_t\}_{t \in [0,T]}$ the completion of $\FF:= \{\fF_t\}_{t \in [0,T]}$ under $\bar \PP$. Consider $\bar \pP$ the predictable $\sigma$-field on $[0,T] \times \bar \VV$ with the filtration $\bar \FF$ on $(\bar \VV, \bB(\bar \VV))$. \\

We make the following assumption on $\nu \in \MM$.
\begin{condition} \label{condition: the measure}
The measure $\nu \in \MM$ is a L\'{e}vy measure on $(\RR^d \backslash \{0\}, \bB(\RR^d \backslash \{0\}))$, i.e. such that
$\int_{0 < |z|<1}  |z|^2 \nu(dz) < \infty$ 
and satisfying
\begin{align} \label{eq: integrability condition measure}
\int_{|z|\geq 1} e^{\alpha |z|^2} \nu(dz) < \infty, \quad \text{ for some } \alpha >1.
\end{align}
\end{condition}
\begin{remark}
\no We stress that the assumption of Gaussian tails (\ref{eq: integrability condition measure}) is paradigmatic within the use of  weak convergence approach arguments for the derivation of moderate deviations principles for jump processes. It is used in the pioneer work \cite{BDG15} and further extensive follow up works that exist in the literature. The assumption of exponential tails for laws that obey large deviations principles is a classical ansatz in the literature of large deviations principles. We cite as reference the Donsker-Varadhan theorem stated as Theorem 3.34 in the monography \cite{Stroock}. The assumption of Gaussian tails (\ref{eq: integrability condition measure}) for $\nu$ is sufficient to the proof of Lemma \ref{lemma: integrability controls} in the Appendix which turns out to be a technical fundamental intermediary result that is fundamental in the derivation of the moderate deviations principle for $(X^\varepsilon)_{\varepsilon>0}$. This restriction still captures a rich class of L\'{e}vy measures $\nu$ allowing the occurrence of infinitely small jumps as it is exhibited in subsection \ref{subsection: examples}. We refer the reader to \cite{Nishimori} for a discussion of the large deviations principle for symmetric stable processes that uses a very different approach than the one we use.
\end{remark}
\paragraph{The space of the delays and the segment function.} 
Fix now $\tau>0$. Given a path $x:[-\tau,T] \longrightarrow \RR^d$ and $t \geq 0$, we use the notation $x_t$ for the segment path defined as $x_t(\theta):=x(t+\theta)$, $\theta \in [-\tau,0]$. Denote by $C([-\tau,T]; \RR^d)$ the space of continuous paths equipped with the uniform norm. We write $\cC:= C([-\tau,0];\RR^d)$.
Let $\DD([-\tau,T];\RR^d)$ be the space of the c\`{a}dl\`{a}g functions equipped with the topology inherited by the $J_1$-metric known as the Skorokhod topology (cf. Chapter 3-p. 111 in \cite{Billingsley}). We write $\dD:= \DD([-\tau,0];\RR^d)$. The space $\DD([-\tau,T];\RR^d)$ turns out to be Polish under this metric.  We refer the reader to Theorem 12.1 and Theorem 12.2 in \cite{Billingsley} for more details. For any $x \in \DD([-\tau,T];\RR^d)$ we write $||x_t||_\infty:= \displaystyle \sup_{-\tau \leq s \leq t} |x(s)|$, $t \geq 0$.
\bigskip
\subsubsection{The multiscale system} 
For every $T>0$, $\tau>0$ and $\varepsilon>0$ we consider the following system of stochastic differential equations, 
\begin{align} \label{eq: the multiscale SDE}
\begin{cases}
X^\varepsilon(t)&= X^\varepsilon(0)+ \displaystyle \int_0^t a(X^\varepsilon_s, Y^\varepsilon(s))ds + \sqrt{\varepsilon} \int_0^t \sigma(X^{\varepsilon}_s) dB^1(s)  + \varepsilon \displaystyle \int_0^t \int_\XX c(X^\varepsilon_{s-},z) \tilde N^{\frac{1}{\varepsilon}}(ds,dz); \\
Y^\varepsilon(t)&= y + \displaystyle \frac{1}{\varepsilon} \int_0^t f(X^\varepsilon_s, Y^\varepsilon(s)) ds + \frac{1}{\sqrt{\varepsilon}} \int_0^t g(X^\varepsilon_s, Y^\varepsilon(s)) dB^2(s) \\
&+\displaystyle \int_0^t \int_\XX h(X^\varepsilon_{s-}, Y^\varepsilon(s-),z) \tilde N^{\frac{1}{\varepsilon}}(ds,dz), \quad t \in [0,T];
\end{cases}
\end{align}
subject to the initial datum
\begin{align} \label{eq: the multisclae SDE initial datum}
\begin{cases}
X^\varepsilon_0&= \chi \in \cC, \\
Y^\varepsilon(0) &= y \in \RR^k,
\end{cases}
\end{align}
where we write $(B(t))_{t \in [0,T]}= (B^1(t), B^2(t))_{t \in [0,T]}$ with $(B^1(t))_{t \in [0,T]}$ and $(B^2(t))_{t \in [0,T]}$ two independent standard Brownian motions with values in $\RR^d$ and $\RR^k$ respectively.  
 We stress out that the multi-scale system (\ref{eq: the multiscale SDE}) has slow and fast component respectively affected by different Brownian signals in small intensity $\varepsilon$ and by the same jump noise signal also in small intensity $\varepsilon>0$ but accelerated in inverse proportion. While the process $(B^1, B^2)$ is also a BM in the space $\RR^{d \times k}$ due to the independence of each component the same does not hold for Poisson random measures in the respective product space of measures. For this reason it is not clear how to use the weak convergence approach developed in \cite{BDG15} that builds in the derivation of a variational formula for functionals of Poisson random measures established in \cite{BDM11}.
In order to guarantee existence and uniqueness of solution for (\ref{eq: the multiscale SDE}) we assume that its coefficients are deterministic measurable functions $a: \dD \times \RR^n \longrightarrow \RR^d$, $\sigma: \dD \longrightarrow \RR^{d \times d}$, $c: \dD \times \XX \longrightarrow \RR^d$, $f: \dD \times \RR^n \longrightarrow \RR^{n \times n}$, $g: \dD \times \RR^k \longrightarrow \RR^{n \times n}$ and $h: \dD \times \RR^n \times \XX \longrightarrow \RR^n$ satisfying the following.
\begin{condition} \label{condition: assumptions for existence uniqueness solution} 
\begin{itemize}
\item[1.] There exists $L>0$ such that for every $\varphi , \tilde \varphi \in \dD$ and $y, \tilde y \in \RR^n$ the following holds 
\begin{align} \label{eq: condition-Lipschitz}
|a(\varphi,y) - a(\tilde \varphi, \tilde y)| & \leq L \Big ( \displaystyle \sup_{t \in [-\tau,0]} |\varphi(t)- \tilde\varphi(t)| + |y - \tilde y| \Big ) \nonumber \\
|\sigma(\varphi) - \sigma(\tilde \varphi)| &\leq L  \Big (\displaystyle \sup_{t \in [-\tau,0]} |\varphi(t)- \tilde \varphi(t)| \Big ) \nonumber \\
\int_{\XX} |c(\varphi,z) - c(\tilde \varphi, z)| \nu(dz) &\leq L  \Big ( \displaystyle \sup_{t \in [-\tau,0]} |\varphi(t)- \tilde \varphi(t)| \Big )   \nonumber \\
|f(\varphi, y) - f(\tilde \varphi, \tilde y)| &\leq L \Big (\displaystyle \sup_{t \in [-\tau,0]} |\varphi(t)- \tilde \varphi(t)|+ |y - \tilde y| \Big ) \nonumber  \\
|g(\varphi,y) - g(\tilde \varphi, \tilde y)| &\leq L \Big (\displaystyle \sup_{t \in [-\tau,0]} |\varphi(t)- \tilde \varphi(t)| + |y - \tilde y| \Big ) \nonumber \\
\int_\XX |h(\varphi,y,z)- h(\tilde \varphi, \tilde y, z)| \nu(dz) & \leq L \Big ( \displaystyle \sup_{t \in [-\tau,0]} |\varphi(t)- \tilde \varphi(t)|   + |y - \tilde y| \Big ).
\end{align}
\item[2.] The functions $c(0,z), h(0,0,z)$ are in $L^1(\nu)$.
\end{itemize}
\end{condition} 

\begin{rem} \label{remark: sublinear growth}
Hypothesis \ref{condition: assumptions for existence uniqueness solution} implies that the coefficients have linear 
growth; i.e. there exists $L_1>0$ such that, for any $\varphi \in \dD$ and $y \in \RR^n$, 
\begin{align} \label{eq: sublinear growth of the coefficients}
|a(\varphi,y)| & \leq L_1 \Big ( 1+ \displaystyle \sup_{t \in [-\tau,0]} |\varphi(t)|+ |y| \Big ) \nonumber \\
|\sigma(\varphi)| &\leq L_1 \Big ( 1+\displaystyle \sup_{t \in [-\tau,0]} |\varphi(t)| \Big ) \nonumber \\
\int_{\XX} |c(\varphi,z)| \nu(dz) &\leq L_1  \Big (1+ \displaystyle \sup_{t \in [-\tau,0]} |\varphi(t)| \Big ) \nonumber \\
|f(\varphi,y)| &\leq L_1 \Big (1+ \displaystyle \sup_{t \in [-\tau,0]} |\varphi(t)| + |y| \Big ) \nonumber  \\
|g(\varphi,y)| &\leq L_1 \Big (1+ \displaystyle \sup_{t \in [-\tau,0]} |\varphi(t)|+ |y| \Big ) \nonumber \\
\int_\XX |h(\varphi,y,z)| \nu(dz) &\leq L_1 \Big ( 1 +\displaystyle \sup_{t \in [-\tau,0]} |\varphi(t)|+ |y| \Big ) .
\end{align}
\end{rem}

The following assumption on the initial delay segment $\zeta$ given in (\ref{eq: the multisclae SDE initial datum}) is of great importance in the establishment of stable estimates for which we derive (\ref{eq: the controlled averaging principle int}).
\begin{condition} \label{condition: initial delay is Lipschitz}
The function $\chi \in \cC$ is Lipschitz continuous with Lipschitz constant $\lambda>0$, i.e.
\begin{align} \label{eq: initial delay is Lipschitz}
|\chi(\theta_1) - \chi(\theta_2)| \leq \lambda |\theta_1 -\theta_2|, \quad \text{for every } \theta_1, \theta_2 \in [-\tau,0].
\end{align}
\end{condition}

\begin{definition} \label{definition: solution of the multiscale system}
Given $T>0$, $\tau>0$, $\varepsilon>0$, $\zeta \in \cC$ and $y \in \RR^k$ we consider the stochastic basis $(\bar \VV, \bB(\bar \VV), \bar \FF, \PP)$. A strong solution of (\ref{eq: the multiscale SDE}) with initial datum (\ref{eq: the multisclae SDE initial datum}) is a stochastic process $(X^\varepsilon,Y^\varepsilon):= \{ (X^\varepsilon(t), Y^\varepsilon(t))\}_{t \in [-\tau,T]}$ such that $X^\varepsilon_0=\chi$, $Y^\varepsilon(0)=y$, $X^\varepsilon(t)$ is $\fF_0$-measurable for any $t \in [-\tau,0]$, $(X^\varepsilon(t), Y^\varepsilon(t))_{t \in [0,T]}$ is $\bar \FF$-adapted and solves (\ref{eq: the multiscale SDE}) $\PP$-a.s.
\end{definition}

\no  We write $\fF_t= \fF_0$ for any $t \in [-\tau,0]$. For any $t \in [0,T]$ and $\varepsilon>0$ the random variables $X^\varepsilon(t) \in \RR^d$ and $Y^\varepsilon(t) \in \RR^k$ are called slow and fast variables respectively under the scale separation by the parameter $\varepsilon>0$ in the vanishing limit $\varepsilon \ra 0$. We underline that the stochastic differential equation for the slow variable $X^\varepsilon$ lifts the problem to an infinite-dimensional setting due to the dependence of the coefficients in terms of the segment path process. \\

\no Given  $T,\tau>0$, $m \in \NN$ and $\bar \FF:= \{ \bar \fF_t \}_{t \in [0,T]}$ we define the space 
\begin{align*}
\sS^2_{\bar \FF}([-\tau,T]; \RR^k) := \Big  \{ \varphi: \Omega \times [-\tau,T] \longrightarrow \RR^k ~|~  &\varphi \text{ is } \bar  \FF-\text{adapted with c\`{a}dl\`{a}g paths such that } \\
&\EE \Big [ \displaystyle \sup_{-\tau \leq u \leq T} |\varphi(u)|^2 \Big ] < \infty \Big \}.
\end{align*}
\no The existence and uniqueness of the solution process $(X^\varepsilon(t), Y^\varepsilon(t))_{t \in [-\tau,T]} \in \sS^2_{\bar \FF}([-\tau,T]; \RR^d) \times \sS^2_{\bar \FF}([- \tau,T]; \RR^n)$  of (\ref{eq: the multiscale SDE}) with initial data (\ref{eq: the multisclae SDE initial datum}) follows from Lemma V.2 and Theorem V.7 of \cite{Protter}, using the convention that $Y^\varepsilon(t)=y$ for all $t \in [- \tau,0]$. This is the content of the following result.
\begin{theorem} \label{theorem: existence uniqueness solution}
Fix $T, \tau, \varepsilon>0$ and $y \in \RR^k$. Let us assume that Hypotheses \ref{condition: the measure},  \ref{condition: assumptions for existence uniqueness solution} and \ref{condition: initial delay is Lipschitz} hold for some $\nu \in \MM$ and $\chi \in \cC$. Then there exists a stochastic process 
$$(X^\varepsilon(t), Y^\varepsilon(t))_{t \in [-\tau,T]}\in \sS^2_{\bar \FF}([-\tau,T]; \RR^d) \times \sS^2_{\bar \FF}([-\tau,T]; \RR^n)$$ that solves uniquely (\ref{eq: the multiscale SDE}) in the sense of Definition (\ref{definition: solution of the multiscale system}).  
\end{theorem}
\bigskip
\subsubsection{The averaged dynamics}
\no We make the further dissipativity  and boundedness assumptions on the coefficients of (\ref{eq: the multiscale SDE}) that yield the existence and uniqueness of solution for the averaged dynamics given by (\ref{eq: int the averaged dynamics}) and some stable a-priori estimates that will be crucial in the derivation of the result announced in the Introduction.
\begin{condition} \label{condition: dissipativity}
\begin{enumerate}
 \item The function $a$ satisfies $a(0,y)=0$ for any $y \in \RR^k$ and there exists $\Lambda>0$ such that
 \begin{align} \label{eq: g and h bounded}
 |g(\zeta ,y)| &\leq \Lambda \nonumber \\
  |h(\zeta,y,z)| & \leq \Lambda |z|, \quad \text{for every } \zeta \in \dD, y \in \RR^k, z \in \XX.
 \end{align}
\item
There exist constants $\beta_1, \beta_2>0$, such that, for any $\zeta, \zeta_1 \in \dD$, $y, \tilde y \times \RR^k$   one has 
\begin{align} \label{eq: dissipativity on the coefficients}
2 \langle y, f(\zeta,y) \rangle + |g(\zeta,y)|^2 + \int_\XX |h(\zeta,y,z)|^2 \nu(dz)  &\leq - \beta_1  |y|^2 + \beta_2 ||\zeta||_\infty^2;
\end{align}
\begin{align} \label{eq: dissipativity Lipschitz on the coefficients}
2 \langle y - \tilde y , f(\zeta,y) - f(\zeta,\tilde y) \rangle + |g(\zeta,y) - g(\zeta, \tilde y)|^2 &+ \int_\XX |h(\varphi,y,z) - h(\varphi, \tilde y,z)|^2 \nu(dz) \nonumber  \\
&\leq - \beta_1 |y - \tilde y|^2 + \beta_2 ||\zeta||^2_\infty
\end{align}
and
\begin{align} \label{eq: dissipativity- the last one required }
2 \langle y - \tilde y , f(\zeta,y) - f(\zeta_1,\tilde y) \rangle \leq - \beta_1 |y- \tilde y|^2 + \beta_2 ||\zeta - \zeta_1||^2_\infty
\end{align}
\end{enumerate}
\end{condition}
\begin{remark}
\no We do not consider a more general framework than Hypotheses \ref{condition: the measure}-\ref{condition: dissipativity} to derive the moderate deviations principle for the family of slow variables $(X^\varepsilon)_{\varepsilon>0}$ from (\ref{eq: the multiscale SDE}). Although  it would be possible to derive the same result under the setting of locally Lipschitz coefficients and the usual weaker local versions of dissipativity conditions stated in Hypothesis \ref{condition: dissipativity} . The reason builds on how the weak convergence approach bypasses the usual verification of exponential tightness through the verification of tightness for controlled modifications of the processes $X^\varepsilon$ under which the use of the usual localization probabilistic techniques work well. Attaining such degree of generality to the cost of a more technical text is beyond the scope of our work.
\end{remark}
\no The following a-priori estimates are straightforward and we omit their proofs.
\begin{proposition} \label{prop: a priori bound slow variable}
Fix $T,\tau >0$ and $y \in \RR^k$. Let Hypothesis \ref{condition: the measure}-\ref{condition: dissipativity} hold for some $\nu \in \MM$ and $\chi \in \cC$. There exists a constant $C_1>0$ independent of $\varepsilon>0$ such that  for all $0 < \varepsilon < 1$ we have
\begin{align} \label{eq: a priori bound slow-fast variable}
\bar \EE \Big [ \displaystyle \sup_{- \tau \leq t  \leq T}  |X^\varepsilon(t)|^2 \Big ]  + \displaystyle \sup_{0 \leq t \leq T} \bar \EE \Big [ |Y^\varepsilon(t)|^2  \Big ] \leq C_1.
\end{align}
\end{proposition}
\no We consider the equation for the fast variable of (\ref{eq: the multiscale SDE}) whenever the slow component is frozen and given by  $\zeta \in \dD$ in the regime $\varepsilon=1$, i.e. fix $y \in \RR^k$; for every $t \geq 0$ let
\begin{align} \label{eq: the SDE for the fast variable with frozen slow component}
Y^{\zeta,y}(t) = y + \int_0^t f(\zeta, Y^{\zeta,y}(s))ds + \int_0^t g(\zeta, Y^{\zeta,y}(s)) dB^2(s) + \int_0^t \int_{\XX} h(\zeta, Y^{\zeta,y}(s-),z) \tilde N^{\frac{1}{1}}(ds,dz).
\end{align}
\no We assume that Hypotheses \ref{condition: the measure}-\ref{condition: dissipativity} hold. We follow closely \cite{Cerrai, Cerrai2} in the argumentation below.\\
Fixed $\zeta \in \dD$ we define the transition semigroup on the space $\BB_b(\RR^k)$ of the bounded measurable functions associated with the jump difusion defined by the strong solution of (\ref{eq: the SDE for the fast variable with frozen slow component}) by
\begin{align} \label{eq: the transition semigroup for the fast variable}
P^\zeta_t f(y) := \bar \EE [f(Y^{\zeta,y}(t))], \quad t \geq 0, \quad y \in \RR^k.
\end{align}
\no In what follows we discuss the existence and uniqueness of an invariant measure for the family of linear operators $(P^\zeta_t)_{t \geq 0}$, i.e. a probability measure $\mu^\zeta \in \pP(\RR^k, \bB(\RR^k))$ such that
\begin{align} \label{eq: invariant measure}
\int_{\RR^k} P^\zeta_t f(y) \mu^\zeta(dy) = \int_{\RR^k} f(y) \mu^\zeta(dy), \quad t \geq 0, f \in \BB_b(\RR^k)
.\end{align}
\no The dissipativity assumption given in (\ref{eq: dissipativity Lipschitz on the coefficients}) yields some $\cC>0$ such that, for any $T_0 \geq 0$, the following bound holds:
\begin{align} \label{eq: uniform bound for tightness}
\displaystyle \sup_{T \geq T_0} \bar \EE[|Y^{\zeta,y}(T)|^2] \leq \cC  e^{- 2 \beta_1 T} (1 + ||\zeta||_\infty^2 + |y|^2).
\end{align}  
\no The estimate (\ref{eq: uniform bound for tightness}) implies that the family of the laws of the process $\{ \lL(Y^{\zeta,y}(T))\}_{T \geq T_0}$ is tight in $\pP(\RR^k; \bB(\RR^k))$ when $T_0 \ra \infty$. Prokhorov's theorem implies the existence of a weak limit $\mu^\zeta$ as $T_0 \ra \infty$ and an indirect use 
of Krylov-Bogliobov's theorem (Theorem 7.1 in \cite{DaPrato}) asserts that $\mu^\zeta$ is an invariant measure of $(P^\zeta_t)_{t \geq 0}$, in the sense of (\ref{eq: invariant measure}).The setting of asumptions made in Hypotheses \ref{condition: the measure}-\ref{condition: dissipativity} imply that the semigroup  $(P^\zeta_t)_{t \geq 0}$ is irreducible. We refer the reader to Proposition 2.4 in \cite{Qiao}. Proposition 7.5 in \cite{DaPrato} implies that $\mu^\zeta$ is the unique invariant measure.  Due to the estimate (\ref{eq: uniform bound for tightness}) and the definition of $\mu^\zeta$ in (\ref{eq: invariant measure}), the simple application of monotone convergence shows, as in Lemma 3.4. in \cite{Cerrai2}, that there exists $C>0$ such that 
\begin{align} \label{eq: second moment invariant measure}
\int_{\RR^k} |y|^2 \mu^\zeta(dy)  \leq C (1 + ||\zeta||^2_\infty + |y|^2).
\end{align} 
\no For any $\zeta \in \dD$ we can define the averaged mixing coefficient
\begin{align} \label{eq: averaged coefficient}
\bar a(\zeta) := \int_{\RR^k} a(\varphi,y) \mu^\zeta (dy).
\end{align}
\no The proof of the following result concerning the Lipschitz continuity of $\bar a$ is straightforward. It follows in the same way the inequality (3.4) in \cite{Xu2018}.
\begin{proposition} \label{proposition: bar a is Lipschitz continuous}
Fix $T,\tau >0$ and $y \in \RR^k$. Let Hypothesis \ref{condition: the measure}-\ref{condition: dissipativity} hold for some $\nu \in \MM$ and $\chi \in \cC$. Then the function $\bar a$ defined by (\ref{eq: averaged coefficient}) is Lipschitz continuous. 
\end{proposition}
\no Proposition \ref{proposition: bar a is Lipschitz continuous} ensures that the averaged differential equation with initial delay data $\chi \in \cC$,
\begin{align} \label{eq: the averaged ode}
\begin{cases}
\frac{d}{dt} \bar X^{0, \chi}(t) &= \bar a(\bar X^{0, \chi}_t), \\
\bar X^{0, \chi}_0 &= \chi 
\end{cases}
\end{align}
has a unique solution $\bar X^{0, \chi} \in \CC([-\tau,T]; \RR^d)$. \\
\no The following proposition, that reads as a strong mixing property of the averaged coefficient $\bar a$ given by (\ref{eq: averaged coefficient}), plays a crucial role in the establishment of the moderate deviations principle for the family $(X^\varepsilon)_{\varepsilon>0}$ since it is a fundamental ingredient in the proof of the controlled averaging principle (\ref{eq: the controlled averaging principle int}).  The derivation of this ergodic property follows as it is done in Lemma 5.2 of \cite{Xu}.
\begin{proposition} \label{proposition: the averaging property for the averaged coefficient}
Fix $T,\tau >0$ and $y \in \RR^k$. Let Hypothesis \ref{condition: the measure}-\ref{condition: dissipativity} hold for some $\nu \in \MM$ and $\chi \in \cC$.Then there exists some function $\alpha:[0, \infty) \longrightarrow [0, \infty)$ such that $\alpha(T) \ra 0$ as $T \ra \infty$ and satisfying for any $t \in [0,T]$
\begin{align} \label{eq: mixing property of the averaged coefficient}
\bar \EE  \Big | \frac{1}{T} \int_t^{t+T} a(\zeta,Y^{\zeta,y}(s)) ds - \bar a(\zeta) \Big |^2 \leq \alpha(T) (1+ ||\zeta||^2_\infty + |y|^2)
\end{align}
where the averaged coefficient $\bar a $ is defined by (\ref{eq: averaged coefficient}).
\end{proposition}
 
\subsection{The main theorem}
\no We make the further assumption on the averaged coefficient $\bar a$ defined by (\ref{eq: averaged coefficient}).
\begin{condition} \label{condition: condition coeffi MDP- differentiablity}
The function $\bar a: \dD \longrightarrow \RR^d$ is Fr\'{e}chet differentiable and is its Fr\'{e}chet derivative is a Lipschitz function, i.e. there exists some constant $L_2>0$ such that
\begin{align} \label{eq: averaged coefficient derivative is Lipschitz}
|D \bar a(\zeta) - D \bar a (\bar \zeta)| &\leq L_2 \Big (\displaystyle \sup_{- \tau \leq t \leq 0} |\zeta(t) - \bar \zeta(t)| \Big ), \quad \zeta, \bar \zeta \in \dD.
\end{align}
\end{condition}
\no We define $L^2(\nu_T):= \Big \{ g:[0,T] \times \XX \longrightarrow [0,\infty) ~|~  \displaystyle \int_0^T \int_\XX |g(s,z)|^2 \nu(dz) ds < \infty \Big \}.$\\
\no The main result of this work is the content of the next theorem and the reader can find its proof in the next section.
\begin{theorem} \label{theorem: MDP for the slow component }
Fix $T,\tau >0$ and $y \in \RR^k$. Let Hypothesis \ref{condition: the measure}-\ref{condition: condition coeffi MDP- differentiablity} hold for some $\nu \in \MM$ and $\zeta \in \cC$. Let 
$$\gG^0: L^2([0,T]; \RR^d) \times L^2(\nu_T) \longrightarrow \CC([-\tau,T]; \RR^d)$$ such that
\begin{align*}
\gG^0(f,g)= \eta,
\end{align*}
where for every $(f,g) \in L^2([0,T]; \RR^d) \times L^2(\nu_T)$ the function $\eta \in \CC([- \tau,T]; \RR^d)$ solves uniquely the skeleton equation
\begin{align} \label{eq: the skeleton equation}
\begin{cases}
\eta (t)&= \displaystyle \int_0^t D \bar a(\bar X^{0,\zeta}_s) \eta_s ds + \int_0^t \cb{\sigma}(\bar X^{0,\zeta}_s) f(s) ds + \int_0^t \int_\XX c(\bar X^{0, \zeta,}_{s},z) g(s,z) \nu(dz)ds, t \in [0,T] \\
\eta_0&=0.
\end{cases}
\end{align}
and the function $\bar X^{0,\chi} \in C([-\tau,T];\RR^d)$ is the unique solution of (\ref{eq: the averaged ode}).\\

For any $\eta \in C([-\tau,T];\RR^d)$ we denote $$\gG_\eta^0:=\Big  \{ (f,g) \in L^2[0,T] \times L^2 (\nu_T) ~|~\gG^0(f,g)= \eta \Big \}.$$ For any $\varepsilon>0$ let $a(\varepsilon)= \varepsilon^{\frac{1-\theta}{2}}$, for some $\theta \in \Big ( \frac{1}{2}, 1\Big )$.\\

For every $\varepsilon>0$ let $(X^{\varepsilon,\chi,y}(t), Y^{\varepsilon,\chi,y}(t))_{t \in [-\tau,T]}$ be the unique strong solution of (\ref{eq: the multiscale SDE}) with initial condition given by (\ref{eq: the multisclae SDE initial datum}) and
 \begin{align} \label{eq: the deviations }
Z^{\varepsilon, \chi,y}:= \frac{X^{\varepsilon, \chi, y} - \bar X^{0, \chi,y}(t)}{a(\varepsilon)}.
\end{align}
The family $(Z^{\varepsilon, \chi,y})_{\varepsilon>0}$  defined by (\ref{eq: the deviations }) satisfies a large deviations principle with speed $b(\varepsilon) = \varepsilon^\theta \ra 0$ as $\varepsilon \ra 0$ for some $\theta \in \Big ( \frac{1}{2}, 1 \Big )$ and the good rate function
\begin{align} \label{eq: the good rate function MDP}
\II (\eta) = \displaystyle \inf_{(f,g) \in \gG_\eta^0 } \frac{1}{2}  \Big (\int_0^T |f(s)|^2 ds + \int_0^T \int_\XX |g(s,z)|^2 \nu(dz)ds \Big ).
\end{align}
with the convention that the $\inf \emptyset= \infty$. \\
\end{theorem}
\subsection{Examples} \label{subsection: examples}
\paragraph*{Strongly tempered exponentially light L\'{e}vy measures.} Hypothesis \ref{condition: the measure} covers a wide class of L\'{e}vy measures and we point out the following special benchmark cases. 
\begin{enumerate}
\item Our setting covers the simplest case of finite intensity super-exponentially light jump measures given by $\nu(dz)= e^{-\alpha |z|^2}$ for some $\alpha>1$. For every $\varepsilon>0$ the corresponding stochastic process $L^\varepsilon_t := \int_0^t \int_\XX z \tilde N^{\frac{1}{\varepsilon}}(ds,dz)$, $t \geq 0$ is a compensated compound Poisson process.
\item More generally Hypothesis \ref{condition: the measure} covers a class of L\'{e}vy measures that mimics the class of strongly tempered exponentially light measures introduced by Rosi\'{n}ski in \cite{Ros07}, however, with a Gaussian damping in order to satisfy (\ref{eq: integrability condition measure}). For the polar coordinate $r=|z|$ and any $A \in \bB(\XX)$ we define
\begin{align*}
\nu(A) = \int_{\RR^d \backslash \{0\}} \int_0^\infty \textbf{1}_{A}(rz) \frac{e^{- r^2}}{r^{\alpha'+1}}dr R(dz), \quad \alpha' \in (0,2),
\end{align*}
for some measure $R \in \MM$ such that $\int_{\RR^d \backslash \{0\}} |z|^{\alpha'} R(dz)< \infty$. We point out that, for every $\varepsilon>0$, the corresponding L\'{e}vy process $(L^\varepsilon_t)_{t \geq 0}$ differs from the compound Poisson process of the paragraph before not only from the fact that the corresponding jump measure has infinite total mass but also from the fact that although a compound Poisson process with positive jumps has almost surely nondecreasing paths, it does not have paths that are almost surely strictly increasing.  Such measures and its corresponding processes were introduced in \cite{OliveiraHogele} for the study of dynamical features of stochastic equations perturbed by jump accelerated noises obeying the large deviations regime.
\end{enumerate}
 \paragraph*{Invariant measures for the Markov semigroup associated to the fast variable.}
 \begin{enumerate}
 \item For every $\varepsilon>0$ and $t \in [0,T]$ let us consider the multiscale system
 \begin{align*}
 \begin{cases}
 d X^\varepsilon(t)&= a(X^\varepsilon_t, Y^\varepsilon(t))dt + \sqrt{\varepsilon} \sigma(X^\varepsilon_t) dB^1(t), \quad X^\varepsilon_0 = \zeta \in \cC \\
 d Y^\varepsilon(t) &= - \frac{1}{2 \varepsilon} Y^\varepsilon(t) + \frac{1}{\varepsilon} d B^2(t), \quad Y^\varepsilon(0)= y \in \RR,
 \end{cases}
 \end{align*}
 where $B^1$ and $B^2$ are two independent standard Brownian motions with values in $\RR$. We assume that the coefficients $a$ and $\sigma$ satisfy Hypotheses \ref{condition: assumptions for existence uniqueness solution} and \ref{condition: dissipativity}. For any $\chi \in \cC$ satisfying Hypothesis \ref{condition: initial delay is Lipschitz} the invariant measure of the fast variable (decoupled of the slow variable in this case) 
 \begin{align*}
 d Y(t)= - \frac{1}{2} Y(t) + dB^2(t), \quad t \geq 0,
\end{align*}
is given by $\mu(dy)= \frac{1}{\sqrt{2 \pi}} e^{- \frac{y^2}{2}}dy$. Hence the averaged coefficient $\bar a$ is given for any $\zeta \in \dD$ by
 \begin{align*}
\bar a(\zeta)= \frac{1}{\sqrt{2 \pi}} \int_\RR a(\zeta,y) e^{- \frac{y^2}{2}}dy.
  \end{align*}
 The function $\bar a$ satisfies Hypothesis \ref{condition: condition coeffi MDP- differentiablity} if $a$ is $C^1$-Fr\'{e}chet differentiable with respect to the first variable $\zeta$.
 
 \item Fix $T,\tau >0$ and $y \in \RR^k$. Let Hypothesis \ref{condition: the measure}-\ref{condition: condition coeffi MDP- differentiablity} hold for some $\nu \in \MM$ and $\chi \in \cC$. 
 For every $\varepsilon>0$ and $t \geq 0$ let us consider the multiscale system (\ref{eq: the multiscale SDE}) with $d=k=1$. We take $f(\zeta,y)= - f_1(\zeta)y$ and $g(\zeta,y)= g(\zeta)$ for every $\zeta \in \dD$ and $y \in \RR^k$ with $f_1(\zeta)>0$ and $g(\zeta)>0$ for any $\zeta \in \dD$. Fix the L\`{e}vy measure $\nu(dz)=e^{-|z|^2}dz$ and since this is a finite measure we consider the non-compensated Poisson random measure $N^{\frac{1}{\varepsilon}}$ instead of $\tilde N^{\frac{1}{\varepsilon}}$. Fixed $\zeta \in \dD$, the Markov semigroup  of the the fast variable governed by the dynamics
 \begin{align*}
 dY^{\zeta,y}(t)= - f_1(\zeta) Y^{\zeta,y}(t) + g(\zeta) dB^2(t) + \int_{\RR \backslash \{ 0\}} \Big (\frac{g(\zeta)}{\sqrt{f_1(\zeta)}} z - Y^{\zeta,y}(t) \Big ) d N^{1}(ds,dz), \quad t \geq 0,
 \end{align*}
 has a unique invariant distribution given by
 \begin{align*}
 \mu^{\zeta}(dy) = \sqrt{\frac{f_1(\zeta)}{\pi g^2(\zeta)}} e^{- \frac{f_1(\zeta) y^2}{g^2(\zeta)}} dy.
 \end{align*}
 The averaged coefficient $\bar a$, given for any $\zeta \in \dD$ by
\begin{align*}
\bar a(\zeta)= \int_{\RR \backslash \{0 \}} a(\zeta,y) \mu^{\zeta}(dy),
\end{align*}
satisfies Hypothesis \ref{condition: condition coeffi MDP- differentiablity} if $a,f$ and $g$ are $C^1$-Fr\'{e}chet differentiable in order to $\zeta$.  This example was inspired on the examples illustrated in \cite{Kumar17} and illustrates that the class of assumptions we make on the coefficients of (\ref{eq: the multiscale SDE}) is not empty. 
 \end{enumerate}

\section{Proof of the main theorem} \label{section: proof}
\no Through all this section let the standing assumptions made in Theorem \ref{theorem: MDP for the slow component } to hold.
Let
\begin{align} \label{eq:parametrization- the speed a}
a(\varepsilon)= \varepsilon^{\frac{1-\theta}{2}}, \varepsilon>0, \quad \text{for some}  \theta \in \Big ( \frac{1}{2},1 \Big ).  
\end{align} 
 The speed of the MDP  is given by $b(\varepsilon):= \frac{\varepsilon}{a^2(\varepsilon)}= \varepsilon^\theta \ra 0$, as $\varepsilon \ra 0$.
\subsection{The ansatz of the weak convergence approach} \label{section:ansatz}
 \paragraph*{Notation.} We follow extensively the notation introduced by Budhiraja, Dupuis and Ganguly in \cite{BDG15}. \\
\no Let $\bar \aA_{+}$ (resp. $\bar \aA$) be the class of all $(\bB(\XX) \otimes \bar \pP)/ \bB([0, \infty))$ (resp. $(\bB(\XX) \otimes \bar \pP) / \bB(\RR)$)- measurable maps from $[0,T] \times \XX \times \bar \VV$ to $[0, \infty)$ 
(resp. $\RR$). For $\varphi \in \bar \aA_+$ let us  define a counting process $N^\varphi$ on $\XX_T$ by
\begin{align} \label{eq: controlled random measure}
N^\varphi(U \times (0,t])(\bar \omega):= \int_U \int_0^t \int_0^\infty \textbf{1}_{[0, \varphi(x,s)(\bar \omega)]}(r) \bar N(dx, dr, ds), \quad t \in [0,T], U \in \bB(\XX).
\end{align}
 \no One can think of $N^{\varphi}$ as a controlled random measure with $\varphi$ selecting the intensity for the points at location $x$ and time $s$ in a possibly random but non-anticipating way.  When $\varphi(x,s,\bar m)= \theta \in (0, \infty)$ we write $N^\varphi= N^\theta$.  For more details we refer the reader to \cite{BDM11}.\\

\no  Define $\ell:[0, \infty) \longrightarrow [0, \infty)$ by
\begin{align*}
\ell(r)= r \ln r - r + 1, \quad r \in [0, \infty).
\end{align*}
\no For any $\varphi \in \bar \aA_+$ and $t \in [0,T]$ define the quantity
\begin{align*}
L_t(\varphi)(\bar \omega):= \int_0^t \int_\XX  \ell (\varphi(s,z,\bar \omega)) \nu(dz) ds .
\end{align*}
This is a well-defined quantity as an $[0,\infty]$-valued random variable. \\

\no Let $\{K_n\}_{n \in \NN} \subset \XX$ be an increasing sequence of compact sets such that $\bigcup_{n =1}^\infty K_n = \XX$. For each $n \in \NN$ let
\begin{align*}
\bar \aA_{b,n}:= \Big \{& \varphi \in \bar \aA_+ ~|~ \text{ for all } (t, \bar \omega) \in [0,T] \times \bar \VV \\
&  \varphi(t,x,\bar m) \in \Big [ \frac{1}{n}, n \Big ] \text{ if } x \in K_n \text{ and } \varphi(t,x,\bar m)=1 \text{ if } x \in K_n^c \Big \}
\end{align*}
and let $\bar \aA_b := \bigcup_{n \in \NN} \bar \aA_{b,n}$. Considering $\varphi$ as a control that perturbs jump rates away from $1$ when $\varphi \neq 1$ we see that the controls in $\bar \aA_b$ are bounded and perturb only off a compact set where the bounds of the set can depend on $\varphi$. \\

\no Consider now the space of random variables  
\begin{align*}
\pP_2:= \Big \{  \xi:  [0,T] \times \bar \VV \longrightarrow \RR^n ~|~\xi \text{ is } \bar \pP \otimes \bB(\RR^n) \text{ measurable such that } \int_0^T |\xi(s,\omega)|^2 ds < \infty \quad \bar \PP-\text{a.s.} \Big \}
\end{align*}
and set $ \uU= \pP_2 \times \bar \aA_+ $. \\

\no For $\xi \in \pP_2$ define
\begin{align*}
\tilde L_T(\xi)(\bar \omega):= \frac{1}{2} \int_0^T |\xi(\bar \omega,s)|^2 ds, \bar \omega \in \bar \VV.
\end{align*}
For a given random control $u=(\xi, \varphi) \in \uU$ define the energy $\bar L_T(u):=  \tilde L_T(\xi) + L_T(\varphi)$. \\
\no For any $M>0$ let
\begin{align*}
\tilde S^M := \{ f \in L^2( [0,T]; \RR^n) ~|~ \tilde L_T(f) \leq M \}.
\end{align*}
Under the $L^2$-weak topology $\tilde S^M$ is a compact subset of $L^2([0,T]; \RR^n)$. Throughout the rest of this work we consider $\tilde S^M$ endowed with this topology. Also let
\begin{align*}
S^M := \{ g: [0,T] \times \XX \longrightarrow [0, \infty) ~|~L_T(g) \leq M \}.
\end{align*}
For any $M>0$ and under the following identity,
\begin{align*}
S^M \simeq \Big \{ \nu^g_T \in \MM ~|~ \nu_T^g(A) := \int_{A} g(s,z) \nu(dz)ds, \quad A \in \bB([0,T] \times \XX) \Big \},
\end{align*}
when considering the vague topology in $\MM$
the space $S^M$ turns out to be compact. For more details we refer the reader to Lemma 5.1 in \cite{BCD13}.\\
\no For any $\varepsilon>0$ and $M>0$ let us consider the following tightned sublevel sets
\begin{align*}
S^{M}_{+, \varepsilon} &:= \Big  \{ g: [0,T] \times \XX  \longrightarrow [0, \infty) ~|~L_T(g) \leq M a^2(\varepsilon) \Big  \}, \\
S^M_\varepsilon &:= \Big \{h: [0,T] \times \XX  \longrightarrow \RR ~|~ h := \frac{g-1}{a(\varepsilon)}, \varphi \in S^{M}_{+, \varepsilon} \Big \}\\
\text{ and } \tilde S^M_{\varepsilon} &:= \Big \{ f: [0,T] \longrightarrow \RR^n ~|~ \tilde L_T(f) \leq M a^2(\varepsilon) \Big \}.
\end{align*}
\no Define also the random sublevel sets
\begin{align} \label{eq: the random sublevel sets}
\uU^M_{+, \varepsilon} &:=  \Big \{ \varphi \in \bar \aA_b ~|~ \varphi(.,.,\omega) \in S^M_{+, \varepsilon} \quad \bar \PP-a.s. \Big \}, \nonumber \\
\uU^M_{\varepsilon} &:= \Big  \{ \psi \in \bar \aA ~|~ \psi(.,.,\omega) \in S^M_{\varepsilon} \quad \bar \PP-a.s. \Big \} \nonumber \\
\text{ and } \tilde \uU^M_\varepsilon &: = \Big \{ \xi \in \pP_2 ~|~ \xi(.,\omega) \in \tilde S^M_\varepsilon \quad \bar \PP-a.s. \Big \}.
\end{align}
\no We reserve the notation $B_2(R)$ for the closed ball of radius $R>0$ in $L^2(\nu_T)$ and $\tilde B_2(R)$ for the closed ball in $L^2([0,T]; \RR^n)$. \\
\no Fix a given Polish space $\UU$.
Given a measurable map $\gG^0: \WW \times L^2(\nu_T) \longrightarrow \UU$ let us write the set of fixed points of $\eta$ under $\gG^0$,
\begin{align*}
 \SSS[\eta]:= \Big \{ (f,g)  \in \WW \times L^2(\nu_T) ~|~ \eta= \gG^0(f,g) \Big \}
\end{align*} 
and define the quadratic form
\begin{align} \label{eq: weak convergence quadratic form}
\II(\eta):= \displaystyle \inf_{(f, g) \in \SSS[\eta]} \frac{1}{2} \Big ( \int_0^T |f(s)|^2 ds + \int_0^T \int_\XX |g(s,z)|^2 \nu_T(ds,dz) \Big ),  \quad \eta \in \UU.
\end{align}
\begin{rem} \label{remark: on automatic tightness of a weak compact collection in L2}
We note that a collection $\{ \psi^\varepsilon\}_{\varepsilon>0} \subset \bar \aA$ with the property that $\displaystyle \sup_{\varepsilon >0} ||\psi^\varepsilon||_2 \leq M$ $\PP$-a.s. for some $M<\infty$ is regarded as a collection of $B_2(M)$-valued random variables where $B_2(M)$ is equipped with the weak topology on the Hilbert space $L^2(\nu_T)$. Since $B_2(M)$ is weakly compact such  collection of random variables is automatically tight. Suppose $\varphi \in S^M_{+, \varepsilon}$; which, we recall, means that $L_T(\varphi) \leq M a^2(\varepsilon)$. Due to Lemma 3.2. in  \cite{BDG15} there exists $\kappa_2(1) \in (0, \infty)$ independent of $\varepsilon>0$ and such that 
$
\psi \textbf{1}_{ \Big \{ |\psi| \leq \frac{1}{a(\varepsilon)} \Big \}} \in B_2(\sqrt{M \kappa_2(1)}),
$
where $\psi:= \frac{\varphi-1}{a(\varepsilon)}$.
\end{rem}
\no The following set of conditions imply the moderate deviations regime.
\begin{condition} \label{condition: condition for the uniform MDP}
Let $\UU$ be a Polish space. For any $\varepsilon>0$ let $\gG^\varepsilon: \VV \longrightarrow \UU$ and $\gG^0: \WW \times L^2(\nu_T) \longrightarrow \UU$ be measurable maps satisfying the following two conditions.
\begin{itemize}
\item[1.] \textbf{Continuity of the limiting map on the controls.} Suppose $(f_n,g_n), (f,g) \in \tilde S^M \times B_2(M)$ such that $(f_n, g_n) \ra (f,g)$ as $n \ra \infty$. 
Then
\begin{align*}
\gG^0(f_n, g_n) \ra \gG^0(f,g) \quad \text{ as } n \ra \infty.
\end{align*}
\item[2.] \textbf{Weak law for the map under shifts by random tightened controls.} For every $ M< \infty$ let $u^\varepsilon:=(\xi^\varepsilon, \varphi^\varepsilon) \in \tilde \uU^M_\varepsilon \times \uU^M_{+,\varepsilon}$. For some $\beta \in (0,1)$ let us assume that $\psi^\varepsilon \textbf{1}_{\{ |\psi^\varepsilon| < \frac{\beta}{a(\varepsilon)}\}} \Rightarrow \psi$ in $B_2(\sqrt{M \kappa_2(1)})$ where $\psi^\varepsilon:= \frac{\varphi^\varepsilon-1}{a(\varepsilon)}$ and $\frac{1}{a(\varepsilon)}\xi^\varepsilon \Rightarrow \xi$ as $\varepsilon \ra 0$ in the weak topology of $L^2([0,T]; \RR^n)$. 
Then 
\begin{align*}
\gG^\varepsilon \Big ( \sqrt{\varepsilon}B +\int_0^. \xi^\varepsilon(s) ds, \varepsilon N^{\frac{1}{\varepsilon} \varphi^\varepsilon} \Big ) \Rightarrow \gG^0 ( \xi , \psi ), \quad \text{ as } \varepsilon \ra 0.
\end{align*}
\end{itemize}
\end{condition}
\begin{theorem} \label{thm: sufficient cond MDP}
Suppose that for every $\varepsilon>0$ the maps $\gG^\varepsilon:\VV \longrightarrow \UU$ and 
$\gG^0: \WW \times L^2(\nu_T) \longrightarrow \UU$ satisfy the conditions of Hypothesis \ref{condition: condition for the uniform MDP}. Then the family $\{ Z^{\varepsilon}\}_{\varepsilon>0}$ defined by
\begin{align} \label{eq: weak convergence maps family}
Z^{\varepsilon}:= \gG^\varepsilon \Big ( \sqrt{\varepsilon}B, \varepsilon N^{\frac{1}{\varepsilon}} \Big ), \quad \varepsilon>0, \quad \varepsilon >0,
\end{align}
satisfies a large deviations principle with speed $b(\varepsilon) \ra 0$ in $\UU$ with good rate function $\II$ given by (\ref{eq: weak convergence quadratic form}).
\end{theorem}
\no Theorem \ref{thm: sufficient cond MDP} is a particular case of Theorem 9.9 in \cite{Budhiraja book}. In what follows we apply Theorem \ref{thm: sufficient cond MDP} to our setting. \\

\no Let us fix $T>0$, $\tau>0$, $(\zeta,y) \in \cC \times \RR^k$ and for every $\varepsilon>0$ let $(X^{\varepsilon,\zeta,y}(t), Y^{\varepsilon,\zeta,y}(t))_{t \in [-\tau,T]}$ be the unique strong solution of (\ref{eq: the multiscale SDE}) with initial datum (\ref{eq: the multisclae SDE initial datum}). For every $\varepsilon>0$ consider $(Z^\varepsilon)_{\varepsilon>0}$ given by (\ref{eq: the deviations }). Under the standing assumptions made in the beginning of this section, for any $\varepsilon>0$, Yamada-Watanabe's theorem ensures the existence of a measurable map $\gG^\varepsilon: \VV \longrightarrow \DD([-\tau,T];\RR^d)$ such that 
\begin{align} \label{eq: Ito maps for the MDP}
Z^{\varepsilon}:= \gG^\varepsilon(\sqrt{\varepsilon}B, \varepsilon N^{\frac{1}{\varepsilon}}).
\end{align}
\no  We recall that $B=(B^1,B^2)$ is a Brownian motion in $\RR^{d \times k}$ due to the independence of $B^1$ and $B^2$ and for any $\varepsilon>0$ the Poisson random measure $N^{\frac{1}{\varepsilon}}$ is independent of $B^1$ and $B^2$ and hence from $B$ which justify the existence of the Ito map $\gG^\varepsilon$.
\no The proof of Theorem \ref{theorem: MDP for the slow component } consists in checking the conditions (1) and (2) of Hypothesis \ref{condition: condition for the uniform MDP} for $(\gG^\varepsilon)_{\varepsilon>0}$ and $\gG^0: \WW \times L^2(\nu_T) \longrightarrow C([-\tau,T];\RR^d)$, $\gG^0(f,g)= \eta$, with $\eta \in C([-\tau,T];\RR^d)$ defined by the skeleton equation (\ref{eq: the skeleton equation}). Hence Theorem \ref{thm: sufficient cond MDP} allows us to conclude. 

\subsection{The skeleton equations and the compactness condition}
For any $\chi \in \cC$ and $u=(f,g) \in L^2([0,T]; \RR^d) \times L^2(\nu_T)$ let us denote by $\bar Z^u \in \CC([-\tau,T]; \RR^d)$ the unique solution of (\ref{eq: the skeleton equation}). By definition we have  
\begin{align*}
\gG^0(f,g) = \bar Z^u.
\end{align*}
\begin{proposition} \label{prop: first condition MDP}
For every $M < \infty$ one has that the set
\begin{align*}
\KK_M := \Big \{ \gG^0  (  f, g  ) ~|~ (f,g) \in  \tilde B_2(M) \times B_2(M) \Big \}
\end{align*}
is compact in $C([-\tau,T]; \RR^d)$.
\end{proposition}
\begin{rem} \label{remark: the first condition MDP}
 Proposition \ref{prop: first condition MDP}  is implied by the following.  Fix $0 \leq M < \infty$. 
 Let $(f_n,g_n)_{n \in \NN} \subset \tilde B_2(M) \times B_2(M)$ such that $(f_n, g_n) \rightharpoonup (f,g)$ as $n \ra \infty$ weakly. Therefore
\begin{align*}
\gG^0( f_n, g_n) \ra \gG^0(f,g) \quad \text{ as } n \ra \infty. 
\end{align*}
\end{rem}
\no The proof of the sentence of Remark \ref{remark: the first condition MDP} that implies Proposition \ref{prop: first condition MDP} is standard. We refer the reader to Lemma 4.1 in the seminal work \cite{BDG15}.
\subsection{The weak limit of the controlled auxiliary processes}
\subsubsection{The equations for the controlled auxiliary processes.}
\no This section serves the purpose of verifying the second condition in Hypothesis \ref{condition: condition for the uniform MDP} for  $\gG^0$ and the family $\{ \gG^\varepsilon: \VV \longrightarrow \DD([- \tau,T]; \RR^d)\}_{\varepsilon>0}$. For every $\varepsilon>0$ recall the random sublevel sets $\uU^M_{\varepsilon}$ and $\tilde  \uU^M_{+, \varepsilon}$ given by (\ref{eq: the random sublevel sets}) and let $u^\varepsilon:=(\xi^\varepsilon, \varphi^\varepsilon) \in  \uU^M_{\varepsilon} \times \tilde \uU^M_{+, \varepsilon}$. Set $\tilde \varphi^\varepsilon= \frac{1}{\varphi^\varepsilon}$. The definition of $\tilde \varphi^\varepsilon$ makes sense since one has $\varphi^\varepsilon \in \aA_b$ $\bar \PP$-a.s. For any $t \in [0,T]$ we define the $\bar \FF$-martingales 
\begin{align*}
\eE(\xi^\varepsilon)(t) &:= \exp \Big ( \int_0^t \xi^\varepsilon(s) dB(s) - \frac{1}{2} \int_0^t |\xi^\varepsilon(s)|^2 ds \Big ) \quad \text{and} \\
\eE(\tilde \varphi^\varepsilon)(t) &:= \exp \Big ( \int_0^t \int_{\XX} \int_0^{\frac{1}{\varepsilon}} \ln \tilde \varphi^\varepsilon(s,z) \bar N(ds,dz,dr) \Big ) \\
& + \int_0^t \int_{\XX} \int_0^{\frac{1}{\varepsilon}} (- \tilde \varphi^\varepsilon(s,z)+1) ds \nu(dz) dr \Big ).
\end{align*}
\no For every $t \in [0,T]$ let $\bar \eE(u^\varepsilon)(t):= \tilde \eE(\xi^\varepsilon)(t) \eE(\tilde \varphi^\varepsilon)(t)$. Girsanov's theorem stated in the form of Theorem III.3.24 of \cite{Jacod Shiryaev} ensures that $(\bar \eE(u^\varepsilon)(t))_{t \in [0,T]}$ is an $\bar \FF$-martingale. Hence the probability measures defined on $(\bar \VV, \bB(\bar \VV))$ by
\begin{align*}
\QQ^\varepsilon_T(G) := \int_G \bar \eE(u^\varepsilon)(T) d \bar \PP, \quad \text{ for all } G \in \bB(\bar \VV)
\end{align*}
are absolutely continuous with respect to $\bar \PP$. Under $\QQ^\varepsilon_T$ the stochastic process 
\begin{align*}
\tilde B^\varepsilon(t):= B(s) - \int_0^t \xi^\varepsilon(s)ds, \quad t \in [0,T],
\end{align*}
is a standard Brownian motion and $\varepsilon N^{\frac{1}{\varepsilon} \varphi^\varepsilon}$ is an independent random measure with the same law of $\varepsilon N^{\frac{1}{\varepsilon}}$ under $\bar \PP$. We recall that
\begin{align*}
N^{\frac{1}{\varepsilon} \varphi^\varepsilon}((0,t] \times U) := \int_0^t \int_U \int_0^\infty \textbf{1}_{[0, \frac{1}{\varepsilon} \varphi^\varepsilon]}(r) \bar N(ds,dz,dr).
\end{align*} 
\no For every $\varepsilon>0$ and $t \in [0,T]$ we write $\xi^\varepsilon(t)=(\xi^\varepsilon_1, \xi^\varepsilon_2)(t) \in \RR^d \times \RR^k$. For any $(\chi,y) \in \cC \times \RR^k$, we define the slow controlled process $(\xX^\varepsilon(t))_{t \in [0,T]}$ and the fast controlled process $(\yY^\varepsilon(t))_{t \in [0,T]}$ given as the strong solutions of (\ref{eq: Khasminki- controlled X variable}) and respectively (\ref{eq: Khasminkii- controlled Y variable}) with respect to $\bar \PP$ (since $\QQ^\varepsilon_T \ll \bar \PP$).

\no For every $\varepsilon>0$ we define $(\bar \xX^\varepsilon(t))_{t \in [0,T]}$ the fast averaged controlled process as the strong solution under $\bar \PP$ of the controlled stochastic differential equation (\ref{eq: Khasminkii- controlled averaged variable}).

For every $\varepsilon>0$ let
\begin{align} \label{eq: the controlled deviation}
\zZ^\varepsilon := \frac{\xX^\varepsilon - \bar X^0}{a(\varepsilon)} =  \gG^\varepsilon \Big (  \sqrt{\varepsilon}B + \int_0^. \xi^\varepsilon(s)ds, \varepsilon N^{\frac{1}{\varepsilon} \varphi^\varepsilon} \Big )
\end{align}
and respectively
\begin{align} \label{eq: the averaged controlled deviation}
\bar \zZ^\varepsilon := \frac{\bar \xX^\varepsilon - \bar X^0}{a(\varepsilon)}.
\end{align}
\paragraph*{The weak limit for the maps under shifts by random tightened controls.} Let $M < \infty$ and $\beta \in (0,1)$. Let $(\xi^\varepsilon, \varphi^\varepsilon) \in \tilde \uU^M_{\varepsilon} \times \uU^M_{+,\varepsilon}$ such that $\psi^\varepsilon \text{1}_{|\psi^\varepsilon| < \frac{\beta}{a(\varepsilon)}} \Rightarrow \psi$ in $B_2(\sqrt{M \kappa_2(1)})$ where $\psi^\varepsilon:= \frac{\varphi^\varepsilon-1}{a(\varepsilon)}$ and $\frac{1}{a(\varepsilon)} \xi^\varepsilon \Rightarrow \xi$ in $\tilde B_2(M)$. The conclusion in the second statement in Hypothesis \ref{condition: condition for the uniform MDP} for $(\gG^\varepsilon)_{\varepsilon>0}$ and $\gG^0$ reads as $\zZ^\varepsilon \Rightarrow \bar Z$, as $\varepsilon \ra 0$, where $\bar Z \in  \CC([-\tau,T]; \RR^d)$ solves uniquely
\begin{align} \label{eq: the controlled averaged eq for mdp}
\begin{cases}
\bar Z(t)= \displaystyle \int_0^t D \bar a(\bar X^{0, \chi}_s) \bar Z(s) ds + \int_0^t \sigma(\bar X^{0,\chi}_s) \xi(s)ds + \int_0^t \int_\XX c(\bar X^{0,\chi}_{s},z) \psi(s,z) \nu(dz)ds, \quad t \in [0,T]\\
\bar Z(t)=0, \quad t \in [-\tau,0].
\end{cases}
\end{align}
In order to prove that $\zZ^\varepsilon \Rightarrow \bar Z$, as $\varepsilon \ra 0$ we proceed as follows.
\begin{itemize}
\item[1.] This step passes through two intermediary tasks. Firstly one shows that the laws of $(\bar \zZ^\varepsilon)_{\varepsilon>0}$ are tight in $\pP(C([-\tau,T]; \RR^d))$ (since compact sets in the topology generated by the uniform convergence are also compact sets in the Skorokhod topology). Then it follows that there exists $\tilde \zZ \in C([-\tau,T];\RR^d)$ such that $\bar \zZ^\varepsilon \Rightarrow \tilde \zZ$ as $\varepsilon \ra 0$. Passing to the pointwise limit in the equation satisfied by $\bar \zZ^\varepsilon$ and due to the uniqueness of solution of (\ref{eq: the controlled averaged eq for mdp}) we conclude that $\tilde \zZ= \bar Z$.
\item[2.] We prove the following strong (controlled) averaging principle:
\begin{align*}
\displaystyle \lim_{\varepsilon \ra 0} \bar \PP \Big ( \displaystyle \sup_{t \in [0,T]} |\zZ^\varepsilon(t) - \bar \zZ^\varepsilon(t)| > \delta\Big )=0, \quad \text{ for any } \delta>0.
\end{align*}
From the limit above and Theorem 4.1. in \cite{Billingsley}, commonly known as Slutzsky's theorem, we can identify $\bar Z$ as the weak limit of $(\zZ^\varepsilon)_{\varepsilon}$ as $\varepsilon \ra 0$.
\end{itemize}
\subsubsection{A priori estimates and a  localization procedure} 
\no For every $\varepsilon>0$ let $\rR(\varepsilon)>0$ such that $\rR(\varepsilon) \ra \infty$ and $a(\varepsilon) \rR^2(\varepsilon) \ra 0$ as $\varepsilon \ra 0$. For example $\rR(\varepsilon) := \frac{1}{\sqrt[4]{a(\varepsilon)}}$, $\varepsilon>0$, does the job. Consequently $\sqrt{\varepsilon} \rR^2(\varepsilon) \ra 0$  and therefore $\varepsilon \rR^2(\varepsilon) \ra 0$ as $\varepsilon \ra 0$. 
For every $\varepsilon>0$ and this choice of $\rR(\varepsilon)$ we define the $\bar \FF$-stopping times
\begin{align} \label{eq: first exit time tilde X}
\tilde \tau^\varepsilon_{\rR(\varepsilon)} := \inf \{ t \in [0,T] ~|~ \xX^\varepsilon(t) \notin  B_{\rR(\varepsilon)}(0)  \}.
\end{align}
and
\begin{align} \label{eq: first exit time bar X}
\bar \tau^\varepsilon_{\rR(\varepsilon)} := \inf \{ t \in [0,T] ~|~ \bar \xX^\varepsilon(t) \notin B_{\rR(\varepsilon)}(0) \}
\end{align}
\no The following list of propositions and lemmas are fundamental estimates used in the strategy described above to obtain the conclusion that $\zZ^\varepsilon \Rightarrow \bar Z$ as $\varepsilon \ra 0$. 

\begin{proposition} \label{proposition: localization technique to the bounded case}
Let the standing assumptions of Theorem \ref{theorem: MDP for the slow component } to hold. For any $0 <M < \infty$, $(\xi^\varepsilon, \varphi^\varepsilon)_{ \varepsilon>0} \subset \tilde{\uU}^M_{+, \varepsilon} \times  \uU^M_{+, \varepsilon}$, $\rR:(0,1] \longrightarrow (0, \infty)$ such that $\rR(\varepsilon) \ra \infty$ and $a(\varepsilon)\rR^2(\varepsilon) \ra 0$ as $\varepsilon \ra 0$ and $T, \tau>0$ we have the following. Given $(\xX^\varepsilon(s))_{s \in [- \tau,T]}$ defined by (\ref{eq: Khasminki- controlled X variable}) and $(\bar \xX^\varepsilon(s))_{s \in [- \tau ,T]}$ by (\ref{eq: Khasminkii- controlled averaged variable}) there exists $0 < \varepsilon_0 <1$ and a constant $\cC>0$ such that for every $0 <\varepsilon< \varepsilon_0$ the following estimates hold: 
\begin{align} \label{eq: localization probability controlled process}
\bar \PP \Big ( \displaystyle \sup_{0 \leq t \leq T} |\xX^\varepsilon(s)| > \rR(\varepsilon )\Big ) \leq 2 e^{- \frac{1}{2} \rR(\varepsilon)} + \cC \varepsilon \rR(\varepsilon)
\end{align}
and 
\begin{align} \label{eq: localization probability averaged controlled process}
\bar \PP \Big ( \displaystyle \sup_{0 \leq t \leq T} |\bar \xX^\varepsilon(s)| > \rR(\varepsilon )\Big ) \leq 2 e^{- \frac{1}{2} \rR(\varepsilon)} + \cC \varepsilon \rR(\varepsilon)
\end{align}
\end{proposition}
\no The proof follows the same reasoning employed in Lemma 2.1 of \cite{OliveiraHogele}. 
\begin{proposition} \label{proposition: a priori bound controlled processes}
Let $M>0$. Fix a function $\rR: (0,\infty) \longrightarrow (0,\infty)$ satisfying the assumptions of Proposition \ref{proposition: localization technique to the bounded case} and for every $\varepsilon>0$ let $\tilde \tau^\varepsilon_{\rR(\varepsilon)}$ defined by (\ref{eq: first exit time tilde X}). Under the assumptions of Hypotheses \ref{condition: the measure}-\ref{condition: condition coeffi MDP- differentiablity} there exists some $\varepsilon_0>0$ such that the following bound holds:
\begin{align} \label{eq: a priori bound controlled processes}
\Gamma_1(M) := \displaystyle \sup_{0 < \varepsilon < \varepsilon_0}  \displaystyle \sup_{(\xi,\varphi) \in \tilde \uU^M_\varepsilon \times \uU^M_{+, \varepsilon}} \Big ( \bar \EE \Big [ \displaystyle \sup_{-\tau \leq t \leq \tilde \tau^\varepsilon_{\rR(\varepsilon)}} |\xX^{\varepsilon}(t)|^2 \Big ] +  \displaystyle \sup_{-\tau \leq t \leq  T} \bar \EE \Big [ | \yY^{\varepsilon}(t)|^2 \textbf{1}_{ \{  \tau^\varepsilon_{\rR(\varepsilon)}>T \}} \Big ] \Big ) < \infty. 
\end{align}
\end{proposition}
The proof follows from applying sucessfully Ito's formula, BDG inequalities and Lemma \ref{lemma: integrability controls} presented in subsubsection \ref{subsubsection: integrability controls} of the Appendix.

\begin{proposition} \label{proposition: a priori bound averaged  slow variable}
Fix $M>0$, $\rR: (0, \infty) \longrightarrow (0, \infty)$ satisfying the hypotheses of Proposition \ref{proposition: localization technique to the bounded case} and for every $\varepsilon>0$ let $\bar \tau^\varepsilon_{\rR(\varepsilon)}$ be defined by (\ref{eq: first exit time bar X}). Under Hypotheses \ref{condition: the measure}-\ref{condition: condition coeffi MDP- differentiablity} there exists some $\varepsilon_0>0$ such that the following holds:
\begin{align} \label{eq: a priori bound averaged slow variable}
\Gamma_2(M) := \displaystyle \sup_{0 < \varepsilon < \varepsilon_0}  \displaystyle \sup_{(\xi,\varphi) \in \tilde \uU^M_\varepsilon \times \uU^M_{+, \varepsilon}} \bar \EE \Big [ \displaystyle \sup_{-\tau \leq t \leq \bar \tau^\varepsilon_{\rR(\varepsilon)}} |\bar \xX^{\varepsilon}(t)|^2 \Big ] < \infty.
\end{align}
\end{proposition}
 \no The proof of Proposition \ref{proposition: a priori bound averaged  slow variable} follows analogously to the proof of (\ref{eq: a priori bound controlled processes}). For this reason we omit it.\\
 
 \begin{lemma} \label{lemma: a priori bound averaged deviation variable}
 Fix $M>0$, $\rR: (0,\infty) \longrightarrow (0, \infty)$ under the hypotheses of Proposition \ref{proposition: localization technique to the bounded case} and for every $\varepsilon>0$ let $\bar \tau^\varepsilon_{\rR(\varepsilon)}$ defined by (\ref{eq: first exit time bar X}).  Under the assumptions of Hypotheses \ref{condition: the measure}-\ref{condition: condition coeffi MDP- differentiablity} there exists some $\varepsilon_0>0$ such that the following holds:
\begin{align} \label{eq: a priori bound averaged deviation variable}
 \Gamma_3(M) := \displaystyle \sup_{0 < \varepsilon < \varepsilon_0}  \displaystyle \sup_{(\xi,\varphi) \in \tilde \uU^M_\varepsilon \times \uU^M_{+, \varepsilon}} \bar \EE \Big [ \displaystyle \sup_{-\tau \leq t \leq \bar \tau^\varepsilon_{\rR(\varepsilon)}} |\bar \zZ^{\varepsilon}(t)|^2 \Big ] < \infty.
\end{align}
\end{lemma}
\no The proof of (\ref{eq: a priori bound averaged deviation variable}) is straightforward and we omit it.

\subsubsection{Identification of the weak limit}
\no Given  $M < \infty$  and $\varepsilon>0$ let $\xi^\varepsilon \in \tilde \uU^M_\varepsilon$, $\varphi^\varepsilon \in \uU^M_{+,\varepsilon}$ and write  $\psi^\varepsilon:= \frac{\varphi^\varepsilon-1}{a(\varepsilon)}$. Assume that for some $\beta \in (0,1)$ the following convergences (in law) are satisfied
 \begin{align*}
 \psi^\varepsilon \textbf{1}_{\{ |\psi^\varepsilon| \leq \frac{\beta}{a(\varepsilon)}\}} \Rightarrow \psi \quad \text{and} \quad  \frac{1}{a(\varepsilon)} \xi^\varepsilon &\Rightarrow \xi, \quad \text{as }\varepsilon \ra 0. 
\end{align*}
Then the following result holds.
\begin{proposition} \label{proposition: identification weak limit for the averaged deviation}
Let the standing assumptions of Theorem \ref{theorem: MDP for the slow component } to hold for some $\nu \in \MM$ and $\xi \in \cC$. For every $\varepsilon>0$ let  $(\bar \zZ^\varepsilon(t))_{t \in [- \tau,T]}$ be  defined by (\ref{eq: the averaged controlled deviation}). Then the family $(\bar \zZ^\varepsilon, \frac{1}{a(\varepsilon)} \xi^\varepsilon, \psi^\varepsilon \textbf{1}_{\{ |\psi^\varepsilon| \leq \frac{\beta}{a(\varepsilon)}\}})_{\varepsilon>0}$ is tight in $\DD([-\tau,T]; \RR^d) \times \tilde B_2(M) \times B_2(\sqrt{M \kappa_2(1)})$ for some $\beta \in (0,1)$ and $\kappa_2(1)$ given in the Remark \ref{remark: on automatic tightness of a weak compact collection in L2}. Furthermore any limit point in law $(\bar Z, \xi, \psi)$ satisfies (\ref{eq: the controlled averaged eq for mdp}).
\end{proposition} 
\no The proof follows with standard arguments used by the weak convergence approach to moderate deviations principles for stochastic differential equations with jumps. We refer the reader to Lemma 4.9 in the seminal work \cite{BDG15}.

\subsection{The controlled averaging principle}
The main result of this section allows us to identify the weak limit of $(\zZ^\varepsilon)_{\varepsilon >0}$ with the weak limit of the family $(\bar \zZ^\varepsilon)_{\varepsilon>0}$ as $\varepsilon \ra 0$.
\begin{theorem} \label{theorem: controlled averaging principle}
 Let the hypotheses of Theorem \ref{theorem: MDP for the slow component } to hold. Then given the families $(\zZ^\varepsilon)_{\varepsilon>0}$ and $(\bar \zZ^\varepsilon)_{\varepsilon>0}$ defined respectively by (\ref{eq: the controlled deviation}) and (\ref{eq: the averaged controlled deviation})  we have for any $\delta>0$ that 
we have
\begin{align} \label{eq: Khasminkii new averaging principle}
\displaystyle \lim_{\varepsilon \ra 0} \PP \Big ( \displaystyle \sup_{0 \leq t \leq \tilde \tau^\varepsilon_{\rR(\varepsilon)}} |\zZ^\varepsilon(t)- \bar \zZ^\varepsilon(t)|> \delta \Big )=0.
\end{align}
\end{theorem} 
The reader can find its proof in Subsubsection \ref{section: proof averagin principle}.
\subsubsection{Khasminkii's auxiliary processes}
\no We follow the technique introduced in \cite{Khasminskii} with the the required modifications to our setting in order to deal with the nonlocal components of the auxiliary processes $(\xX^\varepsilon)_{\varepsilon>0}$ and $(\yY^\varepsilon)_{\varepsilon>0}$ given respectively by (\ref{eq: Khasminki- controlled X variable}) and (\ref{eq: Khasminkii- controlled Y variable}). \\

\no Let $[-\tau,T]$ be divided into intervals of the same length parametrized for every $\varepsilon>0$ 
\begin{align} \label{eq: parametrization- the Delta}
\Delta= \Delta(\varepsilon):= \varepsilon^\gamma a^2(\varepsilon) |\ln \varepsilon|^p, \quad \text{for some } \gamma \in \Big (0, \theta - \frac{1}{2} \Big ) \quad \text{and } p>0.
\end{align}
where the scale $a(\varepsilon)$ is given by (\ref{eq:parametrization- the speed a}). \\

 \no We note the following convergences that follow directly from the choice of $\Delta=\Delta(\varepsilon)$ in (\ref{eq: parametrization- the Delta})  : 
\begin{align} \label{eq: parametrizations- the conv Delta}
&\Delta(\varepsilon) \ra 0; \quad \frac{\Delta(\varepsilon)}{a^2(\varepsilon)} \ra 0; \text{ and }  \quad \frac{\Delta(\varepsilon)}{\varepsilon} \ra \infty \quad \text{ as } \varepsilon \ra 0.
\end{align}
For any $t \in [-\tau,T]$ we denote $t_\Delta := \left \lfloor{\frac{t}{\Delta}} \right \rfloor \Delta$.\\

\no We construct the auxiliary processes $(\hat \yY^\varepsilon(t))_{t \in [0,T]}$ and $(\hat \xX^\varepsilon(t))_{t \in [0,T]}$ by means of the following equations: for any $t \in [0,T]$ let 
\begin{align} \label{eq: Khasminkii -the fast variable with frozen component}
\hat \yY^\varepsilon(t)&= \yY^\varepsilon(t_\Delta) + \frac{1}{\varepsilon} \int_{t_\Delta}^t \Big ( f(\xX^\varepsilon_{t_\Delta}, \hat \yY^\varepsilon(s))+ g(\xX^\varepsilon_{t_\Delta}, \hat \yY^\varepsilon(s)) \xi^\varepsilon_2(s)ds + \int_\XX h( \xX^\varepsilon_{t_\Delta}, \hat \yY^\varepsilon(s),z) (\varphi^\varepsilon(s,z)-1) \nu(dz) \Big ) ds \nonumber \\
& + \frac{1}{\sqrt{\varepsilon}} \int_{t_\Delta}^t g(\xX^\varepsilon_{t_\Delta}, \hat \yY^\varepsilon(s)) d B^2(s) + \int_{t_\Delta}^t \int_\XX h(\xX^\varepsilon_{t_\Delta-}, \hat \yY^\varepsilon(s-),z) \tilde N^{\frac{1}{\varepsilon} \varphi^\varepsilon}(ds,dz)
\end{align}
and 
\begin{align} \label{eq: Khasminkki- the slow variable with frozen component}
\hat \xX^\varepsilon(t) & = \zeta(0) + \int_0^t  \Big (a(\xX^\varepsilon_{s_\Delta},  \yY^\varepsilon(s)) + \sigma(\xX^\varepsilon_s) \xi^\varepsilon_1(s) + \int_\XX c(\xX^\varepsilon_{s},z) (\varphi^\varepsilon(s,z)-1) \nu(dz) \Big ) ds \nonumber \\
& + \sqrt{\varepsilon} \int_0^t \sigma(\xX^\varepsilon_s) dB^1(s) + \varepsilon \int_0^t \int_\XX c(\xX^\varepsilon_{s-},z) \tilde N^{\frac{1}{\varepsilon} \varphi^\varepsilon}(ds,dz).
\end{align}
\subsubsection{Auxiliary estimates.}
\no For every $\varepsilon>0$ let us recall the $\bar \FF$-stopping time $\tilde \tau^\varepsilon_{\rR(\varepsilon)}$ given by (\ref{eq: first exit time tilde X}) for the fixed parametrization $\rR$ given in Proposition \ref{proposition: localization technique to the bounded case}.  The following lemmas are essential a-priori bounds that we use in the proof of the controlled averaging principle stated in Theorem \ref{theorem: controlled averaging principle}. 
\begin{lemma} \label{lemma: Khasminkii segment process estimate}
For every $\varepsilon>0$ let $\rR(\varepsilon)>0$, $b(\varepsilon):= \frac{\varepsilon}{a^2(\varepsilon)}$ and $\Delta(\varepsilon)>0$ fixed as above. Then, for any $(g_\varepsilon)_{\varepsilon>0}$ such that $g(\varepsilon)\simeq_\varepsilon a(\varepsilon)$ as $\varepsilon \ra 0$, the following asymptotic regime holds:
\begin{align} \label{eq: Khasmikkii segment process estimate}
\bar \PP \Big ( \displaystyle \sup_{0 \leq t \leq \tilde \tau^\varepsilon_{\rR(\varepsilon)}} ||\xX^\varepsilon_{t} - \xX^\varepsilon_{t_\Delta}||_\infty > g_\varepsilon \Big ) & \lesssim_\varepsilon  \Xi(\varepsilon) \ra 0, \quad \text{as } \varepsilon \ra 0,
\end{align}
where
\begin{align} \label{eq: Khasminkii-C}
 \Xi(\varepsilon) := b^2(\varepsilon) \frac{\varepsilon}{|\ln \varepsilon|^{2p}} + \frac{\varepsilon^{2 \theta-1 - 2 \gamma}}{|\ln \varepsilon|^{2p -q}} + \frac{\varepsilon^{ 2 \theta - 1 - \gamma}}{|\ln \varepsilon|^{p-q}} + \varepsilon^\gamma \varepsilon^{1- \theta} |\ln \varepsilon|^{2 p}, q > 2 \gamma +3, \quad \varepsilon >0.
\end{align}
\end{lemma}
The proof is given in the subsubsection \ref{subsubsection: segment process} of the Appendix.

\begin{lemma} \label{lemma: Khasminkki}
For every $\varepsilon>0$ let $\rR(\varepsilon)$ fixed as in Proposition \ref{proposition: localization technique to the bounded case} and $\Delta(\varepsilon)$ given by (\ref{eq: parametrization- the Delta}). Then the following convergence holds, 
\begin{align} \label{eq: lemma Khasmkinkii- fast variable convergence}
\displaystyle \sup_{0 \leq t \leq T} \bar \EE \Big  [|\yY^\varepsilon(t) - \hat \yY^\varepsilon(t)| \textbf{1}_{\{T< \tilde \tau^\varepsilon_{\rR(\varepsilon)} \}} \Big ] \lesssim_\varepsilon \frac{C(\varepsilon)}{\Delta(\varepsilon)} e^{- \frac{2 \Delta(\varepsilon)}{2 \varepsilon}+1} \text{as } \varepsilon \ra 0.
\end{align}
for some $C(\varepsilon) \ra 0$ as $\varepsilon \ra 0$ uniformly in the initial condition $(\chi, y) \in \cC \times \RR^k$.
\end{lemma}
The proof is given in subsubsection \ref{subsubsection: Khasminkii} of the appendix.

\subsubsection{Khasminkii's technique}
\begin{proposition} \label{prop: Khasminkii averaging controlled-part1}
For any $\delta>0$ we have
\begin{align} \label{eq: Khasminkii limit1}
\displaystyle \limsup_{\varepsilon \ra 0} \bar \PP \Big ( \displaystyle \sup_{0 \leq t \leq \tilde \tau^\varepsilon_{\rR(\varepsilon)}} |\xX^\varepsilon(t) - \hat \xX^\varepsilon(t)| > \frac{\delta a(\varepsilon)}{2}\Big )=0.
\end{align}
\end{proposition}

\begin{proof}
\no The definitions of $(\yY^\varepsilon(t))_{t \in [0,T]}$ and  $(\hat \yY^\varepsilon(t))_{t \in [0,T]}$ given in (\ref{eq: Khasminkii- controlled Y variable}) and (\ref{eq: Khasminkii -the fast variable with frozen component}) respectively combined with Hypothesis \ref{condition: assumptions for existence uniqueness solution} yield for every $\varepsilon>0$ and $t \in [0,T]$ that
\begin{align*}
\hat \xX^\varepsilon(t) - \xX^\varepsilon(t) =\int_0^t \Big ( a(\xX^\varepsilon_{s_\Delta}, \hat \yY^\varepsilon(s)) - a(\xX^\varepsilon_s, \yY^\varepsilon(s))  \Big )ds \\
 \leq L \int_0^t ||\xX^\varepsilon_{s_\Delta}- \xX^\varepsilon_s||_\infty ds + L \int_0^t |\hat \yY^\varepsilon(s)- \yY^\varepsilon(s)| ds .
\end{align*}
\no The asymptotic behaviour (\ref{eq: parametrizations- the conv Delta}) of $\Delta(\varepsilon)>0$ fixed in (\ref{eq: parametrization- the Delta}) combined with Lemma \ref{lemma: Khasminkii segment process estimate}, (\ref{eq: Khamsminskii C2}) and (\ref{eq: Khasminkii C2 conclusion}) of Lemma \ref{lemma: Khasminkki} yield some $C=C(L,T)>0$ such that 
\begin{align*}
\bar \PP \Big ( \displaystyle \sup_{0 \leq t \leq \tilde \tau^\varepsilon_{\rR(\varepsilon)}} |\hat \xX^\varepsilon(t) - \xX^\varepsilon(t)| > \frac{a(\varepsilon)}{2} \Big ) & \leq \bar \PP \Big ( \int_0^{T \wedge \tilde \tau^\varepsilon_{\rR(\varepsilon)}} |a(\xX^\varepsilon_{s_\Delta}, \hat \yY^\varepsilon(s)) - a(\xX^\varepsilon_s, \yY^\varepsilon(s))| ds > \frac{a(\varepsilon)}{2} \Big ) \\
& \leq  \bar \PP \Big ( \displaystyle \sup_{0 \leq t \leq T \wedge \tilde \tau^\varepsilon_{\rR(\varepsilon)}} || \xX^\varepsilon_{t_\Delta} - \xX^\varepsilon_t||_\infty > C a(\varepsilon)\Big ) \\
&+ \bar \PP \Big ( \int_0^T |\hat \yY^\varepsilon(s) - \yY^\varepsilon(s)|^2  \textbf{1}_{\{ T< \tilde \tau^\varepsilon(\rR(\varepsilon))\}} ds > Ca^2(\varepsilon) \Big ) \\
& \lesssim_\varepsilon \Xi(\varepsilon) + \frac{1}{a^2(\varepsilon)} \int_0^T \bar \EE \Big [ |\hat \yY^\varepsilon(s)- \yY^\varepsilon(s)|^2 \textbf{1}_{\{ T< \tilde \tau^\varepsilon_{\rR(\varepsilon)}\}} ds\Big ] \\
&\lesssim_\varepsilon \Xi(\varepsilon) + \frac{C_2(\varepsilon)}{\Delta(\varepsilon) a^2(\varepsilon)} e^{- \frac{\Delta(\varepsilon)}{2\varepsilon} + 1} \ra 0 \text{ as } \varepsilon \ra 0.
\end{align*}

This finishes the proof of (\ref{eq: Khasminkii limit1}).
\end{proof}
\begin{proposition} \label{prop: Khasminkii averaging controlled-part2}
For any $\delta>0$ we have
\begin{align} \label{eq: Khasminkii limit 2}
\displaystyle \limsup_{\varepsilon \ra 0} \bar \PP \Big ( \displaystyle \sup_{0 \leq t \leq \tilde \tau^\varepsilon_{\rR(\varepsilon)}} |\hat \xX^\varepsilon(t) - \bar \xX^\varepsilon(t)|> \frac{\delta a(\varepsilon)}{2}\Big )=0.
\end{align}
\end{proposition}
\begin{proof}
 \no For every $\varepsilon>0$, $t \in [0,T]$, $\zeta \in \dD$, $\xi \in \tilde \uU^M_{+,\varepsilon}$ and $\varphi^\varepsilon \in \uU^M_{+,\varepsilon}$, we define the function
\begin{align*}
b^\varepsilon(\zeta)(t):= \int_0^t \Big ( \sigma(\zeta) \xi_1^\varepsilon(s) + \int_\XX c(\zeta,z) (\varphi^\varepsilon(s,z)-1) \nu(dz) \Big ).
\end{align*}
\no The definitions of $(\xX^\varepsilon(t))_{t \in [0,T]}$ and $(\hat \xX^\varepsilon(t))_{t \in [0,T]}$ given in  (\ref{eq: Khasminki- controlled X variable}) and respectively in  (\ref{eq: Khasminkki- the slow variable with frozen component}) combined with the definition of $b^\varepsilon$ given above imply for every $t \in [0,T]$ and $\varepsilon>0$   the following identity $\bar \PP$-a.s. on the event $\{ T < \tilde \tau^\varepsilon_{\rR(\varepsilon)} \}$:
\begin{align} \label{eq: Khasminkii-final1}
\hat \xX^\varepsilon(t) - \bar \xX^\varepsilon(t)&= \int_0^t \Big ( b^\varepsilon(\hat \xX^\varepsilon_s) - b^\varepsilon(\bar \xX^\varepsilon_s)\Big ) ds \nonumber \\
& + \int_0^t \Big ( a(\xX^\varepsilon_{s_\Delta}, \hat \yY^\varepsilon(s)) - \bar a(\xX^\varepsilon_s) \Big ) ds \nonumber \\
& + \int_0^t \Big ( \bar a(\xX^\varepsilon_s) - \bar a(\hat \xX^\varepsilon_s) \Big ) ds + \int_0^t \Big ( \bar a(\hat \xX^\varepsilon_s) - \bar a(\bar \xX^\varepsilon_s) \Big ) ds \nonumber \\
& + \sqrt{\varepsilon} \int_0^t \Big (\sigma(\xX^\varepsilon_s) - \sigma(\bar \xX^\varepsilon_s) \Big ) dB^1(s)
\nonumber \\
& + \varepsilon \int_0^t \int_\XX \Big (c(\xX^\varepsilon_{s-},z) - c(\bar \xX^\varepsilon_{s-},z) \Big ) \tilde N^{\frac{1}{\varepsilon} \varphi^\varepsilon}(ds,dz).
\end{align}
\no Hypothesis \ref{condition: assumptions for existence uniqueness solution}, Proposition \ref{proposition: bar a is Lipschitz continuous} and (\ref{eq: Khasminkii-final1}) yield some constant $C=C(L,T)>0$ such that on the event $\{ T < \tilde \tau^\varepsilon_{\rR(\varepsilon)}\}$ we have $\bar \PP$-a.s.
\begin{align*}
\displaystyle \sup_{0 \leq s \leq t} |\hat \xX^\varepsilon(s) - \bar \xX^\varepsilon(s)|^2 & \leq C \Big ( \int_0^t \displaystyle \sup_{0 \leq u \leq s} |\hat \xX^\varepsilon(u) - \bar \xX^\varepsilon(u)|^2 ds+ \displaystyle \sup_{0 \leq s \leq t} \Big | \int_0^s \Big ( a(\xX^\varepsilon_{u_\Delta}, \yY^\varepsilon_u) - \bar a(\xX^\varepsilon_u) \Big )du  \Big |^2  \\
&+ \displaystyle \sup_{ t \in [0,T]} |J^\varepsilon_1(t)|^2    \textbf{1}_{\{ T < \tilde \tau^\varepsilon_{\rR(\varepsilon)}\}} + \displaystyle \sup_{t \in [0,T]} |J^\varepsilon_2(t)|^2   \textbf{1}_{\{ T < \tilde \tau^\varepsilon_{\rR(\varepsilon)}\}}  \Big ),
\end{align*}
where for any $\varepsilon>0$ we write
\begin{align*}
\begin{cases}
J^\varepsilon_1(t) & := \sqrt{\varepsilon} \displaystyle \int_0^t  \Big (\sigma(\xX^\varepsilon_s) - \sigma(\bar \xX^\varepsilon_s) \Big ) dB_1(s) \quad \text{and } \\
J^\varepsilon_2(t) &:=   \varepsilon \displaystyle \int_0^t \int_\XX \Big (c(\xX^\varepsilon_{s-},z) - c(\bar \xX^\varepsilon_{s-},z) \Big ) \tilde N^{\frac{1}{\varepsilon} \varphi^\varepsilon}(ds,dz).
\end{cases}
\end{align*}
\no Gronwall's lemma implies for any $\varepsilon>0$ that
\begin{align} \label{eq: Khasminkii-final2}
\displaystyle \sup_{-\tau \leq t \leq T} |\xX^\varepsilon(t) - \xX^\varepsilon(t)|^2 \textbf{1}_{\{ T < \tilde \tau^\varepsilon_{\rR(\varepsilon)}\}}& \leq e^{CT} \Big ( \displaystyle \sup_{0 \leq s \leq t} \Big | \int_0^s \Big ( a(\xX^\varepsilon_{u_\Delta}, \yY^\varepsilon_u) - \bar a(\xX^\varepsilon_u) \Big )du  \Big |^2 \textbf{1}_{\{ T < \tilde \tau^\varepsilon_{\rR(\varepsilon)}\}} \nonumber\\
& + \displaystyle \sup_{0 \leq t \leq T \wedge \tilde \tau^\varepsilon_{\rR(\varepsilon)}} |J_1^\varepsilon(t)|^2 + \displaystyle \sup_{0 \leq t \leq T \wedge \tilde \tau^\varepsilon_{\rR(\varepsilon)}} |J_2^\varepsilon(t)|^2 \Big ).
\end{align}
\no The estimate (\ref{eq: Khasminkii-final2}) yields for any $\delta>0$ 
\begin{align} \label{eq: Khasminkii-final3}
\bar \PP \Big ( \displaystyle \sup_{0 \leq t \leq \tilde \tau^\varepsilon_{\rR(\varepsilon)}} |\hat \xX^\varepsilon(t) - \bar \xX^\varepsilon(t)| > \frac{a(\varepsilon) \delta}{2} \Big ) & \leq \bar \PP \Big ( \displaystyle \sup_{0 \leq s \leq t} \Big | \int_0^s \Big ( a(\xX^\varepsilon_{u_\Delta}, \yY^\varepsilon(u)) - \bar a(\xX^\varepsilon_u) \Big )du  \Big |^2 \textbf{1}_{\{ T < \tilde \tau^\varepsilon_{\rR(\varepsilon)}\}}  > \frac{\delta^2 a^2(\varepsilon) e^{- 2 CT}}{12} \Big ) \nonumber \\
& + \bar \PP \Big ( \displaystyle \sup_{0 \leq t \leq T \wedge \tilde \tau^\varepsilon_{\rR(\varepsilon)}} |J_1^\varepsilon(t)|^2  > \frac{\delta^2 a^2(\varepsilon)  e^{- 2 CT}}{12} \Big ) \nonumber \\
&+ \bar \PP \Big (  \displaystyle \sup_{0 \leq t \leq T \wedge \tilde \tau^\varepsilon_{\rR(\varepsilon)}} |J_2^\varepsilon(t)|^2  \frac{\delta^2 a^2(\varepsilon)  e^{- 2 CT}}{12} \Big ).
\end{align}
\no Burkholder-Davis-Gundy's inequalities and the sublinear growth of $\sigma$ given by (\ref{eq: sublinear growth of the coefficients}) in Remark \ref{remark: sublinear growth} yield some constant $C_2=C_2(\delta, C_1,L_1, \Gamma_1, \Gamma_2)>0$, where $\Gamma_1, \Gamma_2$ are given by (\ref{eq: a priori bound controlled processes}) in Proposition \ref{proposition: a priori bound controlled processes} and respectively (\ref{eq: a priori bound averaged slow variable}) in Proposition \ref{proposition: a priori bound averaged  slow variable} such that
\begin{align}  \label{eq: Khasminkii-final3A}
 \bar \PP \Big ( \displaystyle \sup_{0 \leq t \leq T \wedge \tilde \tau^\varepsilon_{\rR(\varepsilon)}} |J_1^\varepsilon(t)|^2  > \frac{\delta^2 a^2(\varepsilon)  e^{- 2 CT}}{12} \Big ) &  \leq \frac{12  e^{2 CT}}{\delta^2 a^2(\varepsilon)} \bar \EE \Big [ \displaystyle \sup_{0 \leq t \leq T \wedge \tilde \tau^\varepsilon_{\rR(\varepsilon)}} |J_1^\varepsilon(t)|^2 \Big ] \nonumber \\
 & \leq C_2 \frac{\varepsilon}{a^2(\varepsilon)} \ra 0 , \quad \text{ as } \varepsilon \ra 0. 
\end{align}
\no Analogously, due to Burkholder-Davis-Gundy's inequalities and (\ref{eq: lemma integrability controls- estimate 1}) given in Lemma \ref{lemma: integrability controls} of Appendix-Subsection \ref{subsection: aux results MDP} there exists some constant $C_3=C_3(\delta, C_1, L_1, \Gamma_1, \Gamma_2,M)>0$, that  may change from line to line, such that
\begin{align}  \label{eq: Khasminkii-final3B}
 \bar \PP \Big ( \displaystyle \sup_{0 \leq t \leq T \wedge \tilde \tau^\varepsilon_{\rR(\varepsilon)}} |J_2^\varepsilon(t)|^2  > \frac{\delta^2 a^2(\varepsilon)  e^{- 2 CT}}{12} \Big ) &  \leq \frac{12  e^{ 2 CT}}{\delta^2 a^2(\varepsilon)} \bar \EE \Big [ \displaystyle \sup_{0 \leq t \leq T \wedge \tilde \tau^\varepsilon_{\rR(\varepsilon)}} |J_2^\varepsilon(t)|^2 \Big ] \nonumber \\
 & \leq \frac{\varepsilon}{a^2(\varepsilon)} C_3 \displaystyle \sup_{g \in \sS^M_{+,\varepsilon}} \int_0^T \int_{\XX} |z|^2 g(s,z) \nu(dz) ds \nonumber \\
 & \leq C_3 b(\varepsilon) (T+ a^2(\varepsilon)) \ra 0.
\end{align}
\no  We estimate now the first term in the right hand-side of (\ref{eq: Khasminkii-final3}). For every $\varepsilon>0$ and $t \in [0,T]$ we write $\bar \PP$-a.s. on the event $\{ T < \tilde \tau^\varepsilon_{\rR(\varepsilon)}\}$ 
 \begin{align} \label{eq: Khasminkii-final4}
 \int_0^t \Big ( a(\xX^\varepsilon_{s_\Delta}, \hat \yY^\varepsilon(s)) - \bar a(\xX^\varepsilon_s)) ds \Big ) &= \sum_{k=0}^{\left \lfloor{\frac{t}{\Delta}} \right \rfloor-1} \int_{k \Delta}^{(k+1)\Delta} \Big (a(\xX^\varepsilon_{k \Delta}, \yY^\varepsilon(s)) - \bar a(\xX^\varepsilon_{k \Delta}) \Big ) ds \nonumber \\
 & + \sum_{k=0}^{\left \lfloor{\frac{t}{\Delta}} \right \rfloor -1} \int_{k \Delta}^{(k+1) \Delta} \Big ( \bar a(\xX^\varepsilon_{k \Delta}) - \bar a(\xX^\varepsilon_s) \Big ) ds \nonumber \\
 &+ \int_{t_\Delta}^t \Big ( a(\xX^\varepsilon_{s_\Delta}, \hat \yY^\varepsilon(s)) - \bar a(\xX^\varepsilon_s) \Big ) ds \nonumber \\
 & := I^\varepsilon_1 + I^\varepsilon_2 + I^\varepsilon_3.
 \end{align}
 It follows from (\ref{eq: Khasminkii-final4}) that
 \begin{align} \label{eq: Khasminkii-final5}
  \bar \PP \Big ( \displaystyle \sup_{0 \leq s \leq t} \Big | \int_0^s \Big ( a(\xX^\varepsilon_{u_\Delta}, \yY^\varepsilon_u) - \bar a(\xX^\varepsilon_u) \Big )du  \Big |^2 \textbf{1}_{\{ T < \tilde \tau^\varepsilon_{\rR(\varepsilon)}\}}  > \frac{\delta^2 a^2(\varepsilon)  e^{- 2 CT}}{12} \Big ) & \leq  \bar \PP \Big ( \displaystyle \sup_{0 \leq t \leq T} |I^\varepsilon_1(t)| \textbf{1}_{\{T < \tilde \tau^\varepsilon_{\rR(\varepsilon)}\}} > \frac{\delta a(\varepsilon) e^{- CT}}{6 \sqrt{3} }\Big )  \nonumber \\
 & +  \bar \PP \Big ( \displaystyle \sup_{0 \leq t \leq T} |I^\varepsilon_2(t)| \textbf{1}_{\{T < \tilde \tau^\varepsilon_{\rR(\varepsilon)}\}} > \frac{\delta a(\varepsilon) e^{-  CT}}{6 \sqrt{3}   }\Big ) \nonumber \\
 &  + \bar \PP \Big ( \displaystyle \sup_{0 \leq t \leq T} |I^\varepsilon_3(t)| \textbf{1}_{\{T < \tilde \tau^\varepsilon_{\rR(\varepsilon)}\}} > \frac{\delta a(\varepsilon) e^{-  CT}}{6 \sqrt{3}}\Big ) .
\end{align}
\paragraph*{Estimating $I^\varepsilon_2$.} We observe that for any $\varepsilon>0$ 
\begin{align*}
I^\varepsilon_2 = \int_0^{t_\Delta} \Big ( \bar a(\xX^\varepsilon_{s_\Delta}) - \bar a(\xX^\varepsilon_s)\Big ) ds.
\end{align*}
Proposition \ref{proposition: bar a is Lipschitz continuous} and Lemma \ref{lemma: Khasminkii segment process estimate} implies for some $C_4=C(T)>0$, any $\delta>0$ and $\varepsilon>0$ small enough that
\begin{align} \label{eq: Khasminkii-final I2}
 \bar \PP \Big ( \displaystyle \sup_{t \in [0,T]} |I^\varepsilon_1(t)| \textbf{1}_{\{T < \tilde \tau^\varepsilon_{\rR(\varepsilon)}\}} > \frac{\delta a(\varepsilon) e^{-  CT}}{6 \sqrt{3 } }\Big )  & \leq \bar \PP \Big ( \displaystyle \sup_{0 \leq t \leq \tilde \tau^\varepsilon_{\rR(\varepsilon)}} \int_0^{t_\Delta} |\xX^\varepsilon_{s_\Delta}- \xX^\varepsilon_s| > C_4 a(\varepsilon)\Big ) \lesssim_\varepsilon \Xi (\varepsilon) \ra 0 \quad \text{as } \varepsilon \ra 0.
\end{align}

\paragraph*{Estimating $I^\varepsilon_3$.} Hypothesis \ref{condition: assumptions for existence uniqueness solution}, Proposition \ref{proposition: bar a is Lipschitz continuous} and Proposition \ref{proposition: a priori bound controlled processes} yield some constant $C_5=C_5(L,\Gamma_1(M))>0$ that may change from line to line such that, for every $\varepsilon>0$ small enough and any $\delta>0$, one has
\begin{align} \label{eq: Khasminkii-final I3}
\bar \PP \Big ( \displaystyle \sup_{t \in [0,T]} |I^\varepsilon_3(t)| \textbf{1}_{\{ T < \tilde \tau^\varepsilon_{\rR(\varepsilon)}\}} > \frac{\delta a(\varepsilon)e^{- CT}}{6 \sqrt{3} } \Big ) 
& \leq  \frac{C_5}{a^2(\varepsilon)} \bar \EE \Big [ \displaystyle \sup_{0 \leq t \leq \tilde \tau^\varepsilon_{\rR(\varepsilon)}} \Big | \int_{t_\Delta}^t \Big ( a(\xX^\varepsilon_{s_\Delta}, \hat \yY^\varepsilon(s)) - \bar a(\xX^\varepsilon_s) \Big ) ds \Big |^2 \Big ] \nonumber \\
& \leq \frac{C_5 \Delta(\varepsilon) }{a^2(\varepsilon)} \bar \EE \Big [ \int_0^T \Big ( 1+ ||\xX^\varepsilon_s||_\infty^2 + ||\xX^\varepsilon_{s_\Delta}||^2_\infty + |\yY^\varepsilon(s)|^2 \Big ) \textbf{1}_{\{ T < \tilde \tau^\varepsilon_{\rR(\varepsilon)}\}} ds \Big ]  \nonumber \\
& \lesssim_\varepsilon \frac{\Delta(\varepsilon)}{a^2(\varepsilon)} \ra 0, \quad \text{as } \varepsilon \ra 0,
\end{align}
due to (\ref{eq: parametrizations- the conv Delta}).

\paragraph*{Estimating $I^\varepsilon_1$.} 
\no We construct a new process $Z:= \yY^{\varepsilon} (\xX^\varepsilon_{k \Delta}, \yY^{\varepsilon}(k \Delta))$ where the notation that is displayed here stresses out that the process is the fast variable process $\yY^\varepsilon$ with frozen slow component $\xX^\varepsilon_{k \Delta}$ and initial condition $\yY^{\varepsilon}(k \Delta))$. It is a classical  fact in the course of the Khasminkii's technique employed in \cite{Khasminskii} for the proof of the strong averaging principle that 
for every $s \in [0, \Delta]$ we have
\begin{align*}
(\xX^\varepsilon_{k \Delta}, \yY^\varepsilon(s+k \Delta))=^d \Big ( \xX^\varepsilon_{k \Delta}, \yY^\varepsilon  ( \xX^\varepsilon_{k \Delta}, \yY^\varepsilon(k \Delta)  ) \Big ( \frac{s}{\varepsilon}\Big )\Big ).
\end{align*}
We may assume in addition that the fabricated noises above are independent of $\xX^\varepsilon_{k \Delta}$ and $\yY^\varepsilon(k \Delta)$. For the proof of the statements above we refer the reader to Section 5 in \cite{Xu}.
\no Hence Proposition \ref{proposition: the averaging property for the averaged coefficient} together with the Markov property of $(X^\varepsilon_t, Y^\varepsilon(t))_{t \in [0,T]}$ implies for every $k=0,\dots, \left \lfloor{\frac{t}{\Delta}} \right \rfloor$ the following:
\begin{align} \label{eq: Khasminkii-final8}
\bar \EE \Big [ \Big | \int_{k \Delta}^{(k+1) \Delta} \Big (a(\xX^\varepsilon_{k \Delta}, \hat \yY^\varepsilon(s)) - \bar a(\xX^\varepsilon_{k \Delta}) \Big ) ds \Big | \Big ] 
& \leq  \Delta \bar \EE \Big [ \frac{\varepsilon}{\Delta} \Big | \int_0^{\frac{\Delta}{\varepsilon}} \Big ( a(\xX^\varepsilon_{k \Delta}, Z(s)) - \bar a(\xX^\varepsilon_{k \Delta}) \Big ) ds \Big | \Big ] \nonumber \\
& = \Delta \bar \EE \Big [ \bar \EE \Big [ \Big | \frac{\varepsilon}{\Delta} \int_0^{\frac{\Delta}{\varepsilon}} a(\zeta, Z^{\zeta,y}) - \bar a(\zeta) \Big | \Big | (\zeta,y)=(\xX^\varepsilon_{k \Delta}, \yY^\varepsilon(k \Delta))\Big ]\Big ] \nonumber \\
& \leq \Delta \alpha \Big ( \frac{\Delta}{\varepsilon}\Big ) \Big ( 1 + \bar \EE ||\xX^\varepsilon_{k \Delta}|| + \bar \EE [|\yY^\varepsilon(k \Delta)|]\Big ).
\end{align}
\no Proposition \ref{proposition: the averaging property for the averaged coefficient}, Proposition \ref{proposition: a priori bound controlled processes}, (\ref{eq: parametrizations- the conv Delta}) and (\ref{eq: Khasminkii-final8}) yield, for any $\delta>0$ and $\varepsilon>0$ sufficiently small, that
\begin{align} \label{eq: Khasminkii-final9}
\bar \PP \Big ( \displaystyle \sup_{0 \leq t \leq T} |I^\varepsilon_1| \textbf{1}_{\{ T < \tilde \tau^\varepsilon_{\rR(\varepsilon)}\}} > \frac{\delta a(\varepsilon) e^{- CT}}{6 \sqrt{3}}\Big ) 
& \lesssim_\varepsilon \frac{1}{a^2(\varepsilon)} \bar \EE \Big [ \displaystyle \sup_{0 \leq t \leq T} |I^\varepsilon_1(t)|^2 \textbf{1}_{\{ T < \tilde \tau^\varepsilon_{\rR(\varepsilon)}\}}\Big ] \nonumber \\
& \lesssim_\varepsilon \frac{1}{a^2(\varepsilon)} \sum_{k=0}^{ \left \lfloor{\frac{T}{\Delta(\varepsilon)}} \right \rfloor} \Big ( \bar \EE \Big | \int_{k \Delta}^{(k+1)\Delta} (a(\xX^\varepsilon_{k \Delta}, \hat \yY^\varepsilon(s)) - \bar a(\xX^\varepsilon_{k \Delta})) \textbf{1}_{\{ T < \tilde \tau^\varepsilon_{\rR(\varepsilon)}\}} ds \Big | \Big )^2  \nonumber\\
& \lesssim_\varepsilon \frac{\Delta(\varepsilon)}{a^2(\varepsilon)} \alpha \Big ( \frac{\Delta}{\varepsilon}\Big ) \ra 0 \text{ as } \varepsilon \ra 0.
\end{align}
The convergence above follows from the choice of the parametrization $\Delta= \Delta(\varepsilon)$ fixed in (\ref{eq: parametrization- the Delta}) and $\alpha$ constructed in Proposition \ref{proposition: the averaging property for the averaged coefficient}

\end{proof}

\subsubsection{Proof of Theorem \ref{theorem: controlled averaging principle}} \label{section: proof averagin principle}
\no For any $\varepsilon>0$ fix $\rR(\varepsilon)>0$ such as in Proposition \ref{proposition: localization technique to the bounded case} and recall the definition of $\tilde \tau^\varepsilon_{\rR(\varepsilon)}$ in (\ref{eq: first exit time tilde X}). \\

\no For any $\delta>0$ we have
\begin{align} \label{eq: Khasminkii- the initial estimate}
\displaystyle \limsup_{\varepsilon \ra 0} \bar \PP \Big ( \displaystyle \sup_{0 \leq t \leq T} |\zZ^\varepsilon(t) - \bar \zZ^\varepsilon(t)| > \delta \Big ) & \leq \displaystyle \limsup_{\varepsilon \ra 0} \bar \PP \Big ( \displaystyle \sup_{0 \leq t \leq T} |\xX^\varepsilon(t) - \bar \xX^\varepsilon(t)| > \delta a(\varepsilon) \Big ) \nonumber \\
&\leq  \displaystyle \limsup_{\varepsilon \ra 0} \bar \PP \Big ( \displaystyle \sup_{0 \leq t \leq \tilde \tau^\varepsilon_{\rR(\varepsilon)}} |\xX^\varepsilon(t) - \hat \xX^\varepsilon(t)| > \frac{\delta a(\varepsilon)}{2}  \Big ) \nonumber \\
& + \displaystyle \limsup_{\varepsilon \ra 0} \bar \PP \Big ( \displaystyle \sup_{0 \leq \leq \tilde \tau^\varepsilon_{\rR(\varepsilon)}} |\hat \xX^\varepsilon(t) - \bar \xX^\varepsilon(t)| > \frac{\delta a(\varepsilon)}{2}  \Big ) \nonumber \\
&+ \displaystyle \limsup_{\varepsilon \ra 0} \bar \PP \Big ( \tilde \tau^\varepsilon_{\rR(\varepsilon)} \leq T \Big ) \nonumber \\
& =0,
\end{align}
due to Proposition \ref{proposition: localization technique to the bounded case}, Proposition \ref{prop: Khasminkii averaging controlled-part1} and Proposition \ref{prop: Khasminkii averaging controlled-part2}.
\begin{flushleft}
\qed
\end{flushleft}
\subsection{Conclusion}

\paragraph*{Conclusion-Proof of Theorem \ref{theorem: MDP for the slow component }}  \label{section: proof main result}
We recall the collection of measurable maps $(\gG^\varepsilon)_{\varepsilon>0}$ introduced in (\ref{eq: Ito maps for the MDP}) and $\gG^0$ defined by means of the skeleton equation (\ref{eq: the skeleton equation}). We note that Proposition \ref{prop: first condition MDP} reads as the Condition 1 of Hypothesis \ref{condition: condition for the uniform MDP} for $(\gG^\varepsilon)_{\varepsilon>0}$ and $\gG^0$. Proposition \ref{proposition: identification weak limit for the averaged deviation} combined with Theorem \ref{theorem: controlled averaging principle} yield, due to Slutzky's theorem, that Condition 2 of Hypothesis \ref{condition: condition for the uniform MDP} is verified for $(\gG^\varepsilon)_{\varepsilon>0}$ and $\gG^0$. Hence, the result follows from Theorem \ref{thm: sufficient cond MDP}.
\begin{flushleft}
\qed
\end{flushleft}
\section{Appendix}

\subsection{Auxiliary results for the derivation of the moderate deviations principle} \label{subsection: aux results MDP}
\subsubsection{Integrability properties of the controls} \label{subsubsection: integrability controls}
The following lemma is heavily used in the derivation of the moderate deviations principle stated in Theorem \ref{theorem: MDP for the slow component }. We refer the reader to Subsection \ref{section:ansatz} for notation.
\begin{lemma} \label{lemma: integrability controls}
 Fix $M>0$ and $\nu \in \MM$ a measure satisfying the Hypothesis \ref{condition: the measure}. The following holds.
\begin{itemize}
\item[1.] There exists $\tau>0$ such that for all $\varepsilon>0$ we have
 \begin{align}
&\displaystyle \sup_{g \in S^M_{+,\varepsilon}} \int_I \int_\XX |z|^2 g(s,z) \nu(dz) ds < \tau (a^2(\varepsilon) + |I|), \label{eq: lemma integrability controls- estimate 1} \\
&\displaystyle \sup_{g \in S^M_{+,\varepsilon}} \int_I \int_\XX |z| |g(s,z)-1| \nu(dz) ds  < \tau(a(\varepsilon)+ |I| )  \label{eq: lemma integrability controls- estimate 1.5}
\end{align}
and there exists $\tilde \tau>0$ yielding for all $\varepsilon, \beta >0$ some $c(\beta) \ra 0$ as $\beta \ra \infty$ and such that
\begin{align} \label{eq: lemma integrabilitu controls- estimate 2}
 \displaystyle \sup_{h \in S^M_{ \varepsilon}} \int_I \int_\XX |z| |h(s,z)| \nu(dz) ds < \tilde \tau (\sqrt{|I|} + |I| + a(\varepsilon) + c(\beta)),
\end{align}
for any Borel measurable $I \subset [0,T]$.
\item[2.] For every $\varepsilon>0$ let $\psi^\varepsilon \in \uU^M_{\varepsilon}$. We assume that for some $\beta \in (0,1)$ the following convergence in law holds, $\psi^\varepsilon \textbf{1}_{\{ |\psi^\varepsilon| < \frac{\beta}{a(\varepsilon)}\}} \Rightarrow \psi$ in the compact ball $B_2(\sqrt{M \kappa_2(1)})$, where $\kappa_2(1)$ is given by Remark \ref{remark: on automatic tightness of a weak compact collection in L2}. Then the following convergence in distribution  holds, for every $t \in [0,T]$,
\begin{align} \label{eq: lemma integrabilitu controls- dominated convergence}
\int_0^t \int_{\XX} |z| \psi^\varepsilon(s,z) \nu(dz) ds \ra \int_0^t \int_\XX |z|^r \psi(s,z) \nu(dz) ds.
\end{align}
\end{itemize} 
\end{lemma}  
For the proof of the first statement we refer the reader to Lemma 2.1 in \cite{OliveiraHogele}. The conclusion of the second statement is proved as in Lemma 4.8 of \cite{BDG15}.

\subsection{Auxiliary estimates for the controlled averaging principle}

\subsubsection{Proof of Lemma \ref{lemma: Khasminkii segment process estimate} } \label{subsubsection: segment process}
\no For any $\varepsilon>0$ we fix $\Delta:= \Delta (\varepsilon)$ given by (\ref{eq: parametrization- the Delta}), $a(\varepsilon)$ given in (\ref{eq:parametrization- the speed a}) and $\rR(\varepsilon)>0$ such as in Proposition \ref{proposition: localization technique to the bounded case}. We recall that due to Proposition \ref{proposition: a priori bound controlled processes} we have for any $\varepsilon>0$ small enough that
\begin{align} \label{eq: Khasminkii lemma1-apriori biund revisited}
\displaystyle \sup_{0 < \varepsilon < \varepsilon_0} \bar \EE \Big [ \displaystyle \sup_{0 \leq t \leq \tilde \tau^\varepsilon_{\rR(\varepsilon)}} ||\xX^\varepsilon_t||_\infty \Big ] < \infty,
\end{align}
where $\tilde \tau^\varepsilon_{\rR(\varepsilon)}$ is the $\bar \FF$-stopping time defined by (\ref{eq: first exit time tilde X}).\\

\no  Let us work on the event $\{ T < \tilde \tau^\varepsilon_{\rR(\varepsilon)}\}$. Fix $\varepsilon>0$, $t \in [0,T]$ and  $t_\Delta:= \left \lfloor{\frac{t}{\Delta}} \right \rfloor \Delta$.
For every $\varepsilon>0$ let $K_\varepsilon:=\left \lfloor{\frac{T}{\Delta(\varepsilon)}} \right \rfloor \in \NN $ and $N_\varepsilon:= \left \lfloor{\frac{\tau}{\Delta(\varepsilon)}} \right \rfloor \in \NN$. For any $k=0,\dots, K_\varepsilon-1$ and $m=0,\dots, N_\varepsilon-1$ we label $I^\varepsilon_k:= [k \Delta; (k+1) \Delta]$ and $J^\varepsilon_m:= [-(m+1)\Delta, -m \Delta]$. \\

\no Given $t \in [0,T]$ and $\theta \in [-\tau,0]$ let $k,m \geq 0$ such that $t \in [k \Delta, (k+1) \Delta]$ and $\theta \in [-(m+1) \Delta, -m \Delta]$. It is immediate that
\begin{align*}
t + \theta \in [(k-m-1)\Delta, (k+1-m) \Delta] \quad \text{and } t_\Delta + \theta \in [(k-m-1)\Delta, (k-m)\Delta].
\end{align*}
\no We have to distinguish three possible cases:
\begin{itemize}
\item[(i)] $m \leq k-1$;
\item[(ii)] $m \geq k+1$ and
\item[(iii)] $m=k$.
\end{itemize}
 
\no It follows that
\begin{align*}
\displaystyle \sup_{0 \leq t \leq T} ||\xX^\varepsilon_t - \xX^\varepsilon_{t_\Delta}|| &= \displaystyle \sup_{0 \leq t \leq T} \displaystyle \sup_{-\tau \leq \theta \leq 0} |\xX^\varepsilon(t+\theta)- \xX^\varepsilon(t_\Delta+\theta)| \\
&= \displaystyle \sup_{t \in \displaystyle \cup_{k=0}^{K_\varepsilon-1} I^\varepsilon_k} \quad \displaystyle \sup_{\theta \in \displaystyle \cup_{m=0}^{N_\varepsilon-1} J^\varepsilon_k} |\xX^\varepsilon(t+\theta)- \xX^\varepsilon(t_\Delta+\theta)|.
\end{align*}
\no Let us fix $(g_\varepsilon)_{\varepsilon>0}$ such that $g_\varepsilon \simeq_\varepsilon a(\varepsilon)$ as $\varepsilon \ra 0$. It follows that 
\begin{align*}
\bar \PP \Big ( \displaystyle \sup_{0\leq t \leq T} |\xX^\varepsilon_t - \xX^\varepsilon_{t_\Delta}| > g_\varepsilon; T < \tilde \tau^\varepsilon_{\rR(\varepsilon)} \Big ) &\leq N_\varepsilon K_\varepsilon \displaystyle \max_{\substack{k=0,\dots, K_\varepsilon-1 \\ m=0, \dots, N_\varepsilon-1}}  \bar \PP \Big ( \displaystyle \sup_{ \substack{k \Delta \leq t \leq (k+1) \Delta \\ -(m+1) \Delta \leq \theta \leq -m\Delta}} |\xX^\varepsilon(t+\theta)- \xX^\varepsilon(t_\Delta+\theta)|  > g_\varepsilon; T < \tilde \tau^\varepsilon_{\rR(\varepsilon)} \Big ) \\
&:= K_\varepsilon N_\varepsilon \Big ( p_1^\varepsilon + p_2^\varepsilon+ p_3^\varepsilon \Big ),
\end{align*}
where
\begin{align*}
\begin{cases}
p_1^\varepsilon &:=   \bar \PP \Big ( \displaystyle \sup_{ \substack{k \Delta \leq t \leq (k+1) \Delta \\ -(m+1) \Delta \leq \theta \leq -m\Delta}} |\xX^\varepsilon(t+\theta)- \xX^\varepsilon(t_\Delta+\theta)|  > g_\varepsilon; m \leq k-1; T < \tilde \tau^\varepsilon_{\rR(\varepsilon)} \Big ) \\
 p^\varepsilon_2&:= \bar \PP \Big ( \displaystyle \sup_{ \substack{k \Delta \leq t \leq (k+1) \Delta \\ -(m+1) \Delta \leq \theta \leq -m\Delta}} |\xX^\varepsilon(t+\theta)- \xX^\varepsilon(t_\Delta+\theta)|  > g_\varepsilon; m\geq k+1; T < \tilde \tau^\varepsilon_{\rR(\varepsilon)} \Big ) \text{ and } \\
 p^\varepsilon_3 &:=   \bar \PP \Big ( \displaystyle \sup_{ \substack{k \Delta \leq t \leq (k+1) \Delta \\ -(m+1) \Delta \leq \theta \leq -m\Delta}} |\xX^\varepsilon(t+\theta)- \xX^\varepsilon(t_\Delta+\theta)|  > g_\varepsilon; m=k; T < \tilde \tau^\varepsilon_{\rR(\varepsilon)} \Big ).
\end{cases}
\end{align*}
\paragraph*{Case (i): $m \leq k-1$.} In this case we have that $t + \theta >0$ and $t_\Delta + \theta>0$. Then we have that
\begin{align*}
\xX^\varepsilon(t+\theta)- \xX^\varepsilon(t_\Delta+\theta) &= \int_{t_\Delta+\theta}^{t+\theta} \Big ( a(\xX^\varepsilon_s, \yY^\varepsilon(s)) + \sigma(\xX^\varepsilon_s) \xi^\varepsilon_1(s) + \int_\XX c(\xX^\varepsilon_s,z) (\varphi^\varepsilon(s,z)-1) \nu(dz) \Big ) ds \\
&+ \sqrt{\varepsilon} \int_{t_\Delta+\theta}^{t+\theta} \sigma(\xX^\varepsilon_s)dB_1(s) + \varepsilon \int_{t_\Delta+\theta}^{t+\theta} c(\xX^\varepsilon_{s-},z) \tilde N^{\frac{1}{\varepsilon} \varphi^\varepsilon}(ds,dz).
\end{align*} 
Let us fix the parametrization $L=L_\varepsilon>0$, given by
\begin{align} \label{eq: parametrizations: L}
 L=L_\varepsilon>0:= \frac{a^2(\varepsilon)}{|\ln \varepsilon|^q} \text{ for some } q > 2 \gamma +3, \quad \varepsilon>0.
\end{align}
\no The Bernstein inequality given in the form of Theorem 3.3. in \cite{DZ01} implies for every $\varepsilon>0$ that
\begin{align*}
p^\varepsilon_1 \lesssim_\varepsilon e^{-\frac{g^2_\varepsilon}{L_\varepsilon}} + \PP \Big ( [\xX^\varepsilon - \xX^\varepsilon(t_\Delta+\theta)]_{(k+1)\Delta- m \Delta} > L_\varepsilon; m \leq k-1; T < \tilde \tau^\varepsilon_{\rR(\varepsilon)} \Big ).
\end{align*}
\no  Due to (\ref{eq: Khasminkii lemma1-apriori biund revisited}) it follows for any $\varepsilon>0$ on the event $\{ T < \tilde \tau^\varepsilon_{\rR(\varepsilon)}\}$ that 
\begin{align*}
[\xX^\varepsilon - \xX^\varepsilon(t_\Delta+\theta)]_{(k+1)\Delta- m \Delta} &\lesssim_\varepsilon \varepsilon \Delta(1+ \rR^2(\varepsilon))+ \varepsilon^2 \int_{t_\Delta-(m+1)\Delta}^{(k+1)\Delta - m \Delta} |z|^2 N^{\frac{1}{\varepsilon}}(ds,dz) \\
& := \varepsilon(1+ \rR^2(\varepsilon)) \Delta + \varepsilon^2 I^\varepsilon_{(k+1)\Delta - m \Delta}.
\end{align*}
Due to the choice of $L_\varepsilon$ in (\ref{eq: parametrizations: L}) and $\Delta(\varepsilon)$ in (\ref{eq: parametrization- the Delta}) let $\varepsilon_0>0$ sufficiently small such that for any $\varepsilon< \varepsilon_0$ we have  $\varepsilon(1+ \rR^2(\varepsilon)) \varepsilon^\gamma |\ln \varepsilon|^{p-q} < \frac{1}{2}$. Then it follows that
\begin{align} \label{eq: estimative of Lepsilon}
\varepsilon (1 + \rR^2(\varepsilon)) \Delta = \varepsilon(1 + \rR^2(\varepsilon)) \varepsilon^\gamma |\ln \varepsilon|^{p-q} \frac{a^2(\varepsilon)}{|\ln \varepsilon|^q} < \frac{L_\varepsilon}{2}
\end{align}
for every $\varepsilon< \varepsilon_0$.

\no The estimate (\ref{eq: lemma integrability controls- estimate 1}) in Lemma \ref{lemma: integrability controls} (Subsection \ref{subsection: aux results MDP} of the Appendix) implies for any $\varepsilon>0$ small enough such that (\ref{eq: estimative of Lepsilon}) holds that
\begin{align*}
 \PP \Big ( [\xX^\varepsilon - \xX^\varepsilon(t_\Delta+\theta)]_{(k+1)\Delta- m \Delta} > L_\varepsilon; m \leq k-1; T < \tilde \tau^\varepsilon_{\rR(\varepsilon)} \Big )& \leq \bar \PP \Big ( \varepsilon^2 I^\varepsilon_{(k+1)\Delta - m \Delta} > \frac{L_\varepsilon}{2} \Big ) \\
 & \lesssim_\varepsilon \frac{\varepsilon^2}{2 L_\varepsilon} \bar \EE \Big [ I^\varepsilon_{(k+1)\Delta - m\Delta} \Big ] \\
 & \lesssim_\varepsilon \frac{\varepsilon}{L\varepsilon} \int_{t_\Delta  - (m+1)\Delta}^{(k+1)\Delta - m\Delta} |z|^2 \varphi^\varepsilon(s,z) \nu(dz) ds \\
 & \lesssim_\varepsilon \frac{\varepsilon}{L_\varepsilon} \Big ( a^2(\varepsilon) + \Delta \Big ).
\end{align*}
\no Due to (\ref{eq: parametrizations- the conv Delta}), (\ref{eq: parametrizations: L}), and $a(\varepsilon) \ra 0$ as $\varepsilon \ra 0$ we conclude for every $\varepsilon>0$ small enough that 
\begin{align} \label{eq: Khasminkii lemma1-p1}
p^\varepsilon_1 \lesssim_\varepsilon e^{- \frac{g^2_\varepsilon}{L_\varepsilon}} + \frac{\varepsilon}{L_\varepsilon} \Big ( a^2(\varepsilon) + \Delta (\varepsilon) \Big ) \ra 0 \quad \text{ as } \varepsilon \ra 0.
\end{align}

\paragraph*{The case $m \geq k+1$}.  In this case we have that $t + \theta <0$ and $t_\Delta + \theta<0$. Since the initial delay $\chi$ is Lipschitz continuous (cf. (\ref{eq: initial delay is Lipschitz})) it follows that
\begin{align*}
|\xX^\varepsilon(t+\theta) - \xX^\varepsilon(t_\Delta+\theta)| 
&= |\chi(t+\theta) - \chi(t_\Delta+\theta)|\\
&\leq \lambda |t-t_\Delta|.
\end{align*}
\no Then, for any $\varepsilon>0$ we have 
\begin{align} \label{eq: Khasminkki lemma1-p2}
p^\varepsilon_2 &=\bar \PP \Big ( \displaystyle \sup_{ \substack{k \Delta \leq t \leq (k+1) \Delta \\ -(m+1) \Delta \leq \theta \leq -m\Delta}} |\xX^\varepsilon(t+\theta)- \xX^\varepsilon(t_\Delta+\theta)|^4  > (g_\varepsilon)^4; m\geq k+1; T < \tilde \tau^\varepsilon_{\rR(\varepsilon)} \Big ) \nonumber\\
 &\lesssim_\varepsilon \frac{1}{(g_\varepsilon)^4} \bar \EE \Big [ \sup_{ \substack{k \Delta \leq t \leq (k+1) \Delta \\ -(m+1) \Delta \leq \theta \leq -m\Delta}} |\xX^\varepsilon(t+\theta)- \xX^\varepsilon(t_\Delta+\theta)|^4 \Big ] \lesssim_\varepsilon \Big (\frac{\Delta(\varepsilon)}{g_\varepsilon} \Big )^4 \ra 0 \quad \text{ as } \varepsilon \ra 0,
\end{align}
 due to the definition of $\Delta(\varepsilon)$ in (\ref{eq: parametrization- the Delta}) and $g_\varepsilon \simeq_\varepsilon a(\varepsilon)$ as $\varepsilon \ra 0$. 
\paragraph*{The case $m=k$.} In this case we have $t + \theta \in [-\Delta, \Delta]$ and $t_\Delta +\theta \in [-\Delta, 0]$. It is immediate that
\begin{align*}
|\xX^\varepsilon(t+\theta) - \xX^\varepsilon(t_\Delta+\theta)| = |\xX^\varepsilon(t+\theta) - \xX^\varepsilon(t_\Delta+\theta)|\textbf{1}_{\{ t + \theta>0 \}} + |\xX^\varepsilon(t+\theta) - \xX^\varepsilon(t_\Delta+\theta)| \textbf{1}_{\{ t+\theta<0 \}}.
\end{align*}
Due to the two previous cases already analysed we have, for any $\varepsilon>0$ small enough ,that
\begin{align} \label{eq: Khasminkiii lemma1-p3}
p^\varepsilon_3 & \leq  \bar \PP \Big (\displaystyle \sup_{ \substack{k \Delta \leq t \leq (k+1) \Delta \\ -(k+1) \Delta \leq \theta \leq -k \Delta}} |\xX^\varepsilon(t+\theta)- \xX^\varepsilon(t_\Delta+\theta)|  > g_\varepsilon; \textbf{1}_{\{ t + \theta>0 \}}; T < \tilde \tau^\varepsilon_{\rR(\varepsilon)}  \Big ) \nonumber \\
& + \bar \PP \Big (\displaystyle \sup_{ \substack{k \Delta \leq t \leq (k+1) \Delta \\ -(k+1) \Delta \leq \theta \leq -k \Delta}} |\xX^\varepsilon(t+\theta)- \xX^\varepsilon(t_\Delta+\theta)|  > g_\varepsilon; \textbf{1}_{\{ t + \theta<0 \}}; T < \tilde \tau^\varepsilon_{\rR(\varepsilon)}  \Big ) \nonumber  \\
& \lesssim_\varepsilon e^{- \frac{g^2_\varepsilon}{L_\varepsilon}} + \frac{\varepsilon}{L_\varepsilon} \Big ( a^2(\varepsilon) + \Delta \Big ) +  \Big ( \frac{\Delta}{g_\varepsilon} \Big )^4 \ra 0 \quad \text{ as } \varepsilon \ra 0.
\end{align}
Combining (\ref{eq: Khasminkii lemma1-p1})-(\ref{eq: Khasminkiii lemma1-p3}) it follows, for $\Delta(\varepsilon)$, $L_\varepsilon$ given by (\ref{eq: parametrization- the Delta}) and respectively (\ref{eq: parametrizations: L}) and any $\varepsilon>0$ small enough, that 
\begin{align*}
&\bar \PP \Big ( \displaystyle \sup_{0\leq t \leq T} |\xX^\varepsilon_t - \xX^\varepsilon_{t_\Delta}| > g_\varepsilon; T < \tilde \tau^\varepsilon_{\rR(\varepsilon)} \Big ) \\
& \lesssim_\varepsilon
N_\varepsilon K_\varepsilon \Big ( e^{- \frac{g^2_\varepsilon}{L_\varepsilon}} + \frac{\varepsilon}{L_\varepsilon} \Big ( a^2(\varepsilon) + \Delta \Big ) +  \Big (\frac{\Delta}{g_\varepsilon} \Big )^4 \Big ) \\
& \lesssim_\varepsilon \frac{1}{(\Delta(\varepsilon))^2} \Big ( e^{- |\ln \varepsilon|^q} + \varepsilon |\ln \varepsilon|^q + b(\varepsilon) |\ln \varepsilon|^q \Delta + \Big ( \frac{\Delta(\varepsilon)}{a(\varepsilon)} \Big )^4 \Big ) \\
& \lesssim_\varepsilon \frac{\varepsilon^q}{\varepsilon^{2 \gamma} a^4(\varepsilon) |\ln \varepsilon|^{2p}} + \frac{\varepsilon}{\varepsilon^{2 \gamma} a^4(\varepsilon) |\ln \varepsilon|^{2p-q}} + \frac{\varepsilon}{\varepsilon^\gamma a^4(\varepsilon) |\ln \varepsilon|^{p-q}} + a^2(\varepsilon) \varepsilon^{2 \gamma} |\ln \varepsilon|^{2p} \\
& \lesssim_\varepsilon   b^2(\varepsilon) \frac{\varepsilon}{|\ln \varepsilon|^{2p}} + \frac{\varepsilon^{2 \theta-1 - 2 \gamma}}{|\ln \varepsilon|^{2p -q}} + \frac{\varepsilon^{ 2 \theta - 1 - \gamma}}{|\ln \varepsilon|^{p-q}} + \varepsilon^\gamma \varepsilon^{1- \theta} |\ln \varepsilon|^{2 p} =: \Xi(\varepsilon). 
\end{align*}
\no Since $\gamma \in \Big ( 0, \theta - \frac{1}{2} \Big )$, $a(\varepsilon)= \varepsilon^{\frac{1-\theta}{2}}$, $\theta \in \Big ( \frac{1}{2}, 1 \Big )$, $b(\varepsilon) = \frac{\varepsilon}{a^2(\varepsilon)}$ we conclude that $\Xi(\varepsilon) \ra 0$ as $\varepsilon \ra 0$. This finishes the proof.

\begin{flushright}
\qed
\end{flushright}
\subsubsection{Proof of Lemma \ref{lemma: Khasminkki} } \label{subsubsection: Khasminkii}

\no Ito's formula yields for any $t \in [t_\Delta, t_\Delta+1]$ and $\bar \PP$-a.s.
\begin{align*}
&|\hat \yY^\varepsilon(t) - \yY^\varepsilon(t)|^2 \\
& = \frac{2}{\varepsilon} \int_{t_\Delta}^t \langle f(\xX^\varepsilon_{t_\Delta}, \hat \yY^\varepsilon(s)) - f(\xX^\varepsilon_s, \yY^\varepsilon(s)), \hat \yY^\varepsilon(s)- \yY^\varepsilon(s) \rangle ds \\
& + \frac{2}{\varepsilon} \int_{t_\Delta}^t \langle (g(\xX^\varepsilon_{t_\Delta},\hat  \yY^\varepsilon(s))- g(\xX^\varepsilon_s, \yY^\varepsilon(s)) )\xi^\varepsilon_2(s), \hat \yY^\varepsilon(s)- \yY^\varepsilon(s) \rangle ds \\
&+ \frac{2}{\varepsilon} \int_t^{t_\Delta} \int_\XX \langle h(\xX^\varepsilon_{t_\Delta}, \hat \yY^\varepsilon(s),z)- h(\xX^\varepsilon_s, \yY^\varepsilon(s),z), \hat \yY^\varepsilon(s)- \yY^\varepsilon(s)\rangle (\varphi^\varepsilon(s,z)-1) \nu(dz) ds \\
&+ \frac{2}{\sqrt{\varepsilon}} \int_{t_\Delta}^t \langle g(\xX^\varepsilon_{t_\Delta}, \hat \yY^\varepsilon(s))- g(\xX^\varepsilon_s, \yY^\varepsilon(s)), (\hat \yY^\varepsilon(s)- \yY^\varepsilon(s)) dB^2(s)\rangle \\
& + \frac{1}{\varepsilon} \int_{t_\Delta}^t |g(\xX^\varepsilon_{t_\Delta}, \hat \yY^\varepsilon(s)) - g(\xX^\varepsilon_s, \yY^\varepsilon(s))|^2 ds \\
&+ \int_{t_\Delta}^t \int_\XX 2 \langle h(\xX^\varepsilon_{t_\Delta-}, \hat \yY^\varepsilon_{s-}, z)- h(\xX^\varepsilon_{s-}, \yY^\varepsilon_{s-},z), \yY^\varepsilon_{s-}- \yY^\varepsilon_{s-} \rangle  \tilde N^{\frac{1}{\varepsilon} \varphi^\varepsilon}(ds,dz)  \\
&+  \int_{t_\Delta}^t \int_\XX |h(\xX^\varepsilon_{t_\Delta-}, \hat yY^\varepsilon_{s-}, z)- h(\xX^\varepsilon_{s-}, \yY^\varepsilon_{s-},z), \hat \yY^\varepsilon_{s-}- \yY^\varepsilon_{s-}|^2 \tilde N^{\frac{1}{\varepsilon} \varphi^\varepsilon}(ds,dz) \\
& +\frac{1}{\varepsilon} \int_{t_\Delta}^t \int_\XX |h(\xX^\varepsilon_{t_\Delta}, \hat yY^\varepsilon(s))- h(\xX^\varepsilon_s, \yY^\varepsilon(s))|^2 \varphi^\varepsilon(s,z)\nu(dz)ds \\
&= \sum_{i=1}^8 I^\varepsilon(t).
\end{align*}

\no Using (\ref{eq: dissipativity- the last one required }) in Hypothesis \ref{condition: dissipativity} yields for any $\varepsilon>0$ and $t \in [t_\Delta, t_\Delta+1]$ 
\begin{align}\label{eq: Khasminkii 2nd estimate-countdown3}
I^\varepsilon_1(t) \leq - \frac{2 \beta_1}{\varepsilon} \int_{t_\Delta}^t |\hat \yY^\varepsilon(s)- \yY^\varepsilon(s)|^2  ds + \frac{2 \beta_2 \Delta}{\varepsilon} ||\xX^\varepsilon_{t_\Delta} - \xX^\varepsilon_{t}||^2_\infty.
\end{align}
\no The boundedness of $g$ given by (\ref{eq: g and h bounded}) in Hypothesis \ref{condition: dissipativity}, the fact that $\xi^\varepsilon \in \tilde \uU^M_\varepsilon$ and  Cauchy-Schwartz's inequality imply for any $\varepsilon>0$ and $t \in [t_\Delta, t_\Delta+1]$ that
\begin{align}\label{eq: Khasminkii 2nd estimate-countdown4}
I^\varepsilon_2(t)  \leq \frac{4 \Lambda \sqrt{M} a(\varepsilon)}{\varepsilon} \Big ( 1 + \int_{t_\Delta}^t   |\hat \yY^\varepsilon(s)- \yY^\varepsilon(s)|^2 ds \Big ).
\end{align}
\no Analogously, (\ref{eq: dissipativity- the last one required }) in Hypothesis \ref{condition: dissipativity} together with (\ref{eq: lemma integrability controls- estimate 1}), (\ref{eq: lemma integrability controls- estimate 1.5}) given in Lemma \ref{lemma: integrability controls} of Subsection \ref{subsection: aux results MDP} of the Appendix combined with the numeric fact  $x \lambda \leq x^2+\frac{1}{\lambda}$, $x, \lambda \geq 0$ 
yield some $C_1=C_1(M, \Lambda)>0$ such that for any $\varepsilon>0$ and $t \in [t_\Delta, t_\Delta+1]$ we have
\begin{align}\label{eq: Khasminkii 2nd estimate-countdown5}
 I^\varepsilon_3(t) &\leq \frac{C_1}{\varepsilon \lambda } (a(\varepsilon) + \Delta) +\frac{C_1 \lambda}{\varepsilon} \int_{t_\Delta}^t    |\hat \yY^\varepsilon(s)- \yY^\varepsilon(s)|^2 \Theta^\varepsilon(s)ds \quad \text{ and } \nonumber \\
I^\varepsilon_5(t) + I^\varepsilon_8(t) & \leq \frac{C}{\varepsilon} \Big ( a^2(\varepsilon) + \Delta \Big )
\end{align} 
where $\Theta^\varepsilon(t):= \int_0^t |z||\varphi^\varepsilon(s,z)-1| \nu(dz), \quad t \in [0,T]$. \\
The estimates (\ref{eq: Khasminkii 2nd estimate-countdown3})-(\ref{eq: Khasminkii 2nd estimate-countdown5}) imply for $t \in [t_\Delta, t_\Delta+1]$, $\varepsilon>0$ and $\lambda= \lambda(\varepsilon)>0$ fixed below the following $\bar \PP$-a.s. bound on the event $\{ T < \tilde \tau^\varepsilon_{\rR(\varepsilon)} \}$: 
\begin{align*}
|\hat \yY^\varepsilon(t) - \yY^\varepsilon(t)|^2  \lesssim_\varepsilon \int_{t_\Delta}^t \frac{1}{\varepsilon} \Big (-1 + a(\varepsilon) + \lambda(\varepsilon) \Theta^\varepsilon(s) \Big ) |\hat \yY^\varepsilon(s) - \yY^\varepsilon(s)|^2 ds + C_2(\varepsilon) + I^\varepsilon_4(t) + I^\varepsilon_6(t) + I^\varepsilon_7(t),
\end{align*}
where 
\begin{align} \label{eq: Khamsminskii C2}
C_2(\varepsilon) \simeq_\varepsilon \frac{1}{\varepsilon} \Big ( \Delta(\varepsilon) \rR^2(\varepsilon)+ a(\varepsilon) + \frac{a(\varepsilon)}{\lambda(\varepsilon)} (1 + \Delta(\varepsilon)) + a^2(\varepsilon) + \Delta(\varepsilon) \Big )\quad \text{
 as }\varepsilon \ra 0.
\end{align} 
 
 \no Due to Gronwall's lemma, the estimate (\ref{eq: lemma integrability controls- estimate 1.5}) in Lemma \ref{lemma: integrability controls} (Subsection \ref{subsection: aux results MDP} of the Appendix) and the fact that $\bar \EE[I^\varepsilon_4]= \bar \EE[I^\varepsilon_6]=\bar \EE[I^\varepsilon_7]=0$ it follows, for any $\varepsilon>0$, $\lambda= \lambda(\varepsilon)=\varepsilon$ and $t \in [t_\Delta, t_\Delta+1]$ that
 \begin{align*}
 \bar \EE \Big [ \Big |\hat \yY^\varepsilon(t) - \yY^\varepsilon(t)|^2 \textbf{1}_{\{ T < \tilde \tau^\varepsilon_{\rR(\varepsilon)} \}}] \lesssim_\varepsilon C_2(\varepsilon) \exp \Big ( \frac{-\Delta(\varepsilon)}{\varepsilon} (1 - a(\varepsilon)) + a(\varepsilon) + \Delta(\varepsilon) \Big ).
 \end{align*}
 \no Let $\varepsilon_0>0$ small enough such that $1 - a(\varepsilon)> \frac{1}{2}$ and $a(\varepsilon) + \Delta(\varepsilon) < 1$ for any $\varepsilon< \varepsilon_0$. Therefore we have for any $\varepsilon>0$ small enough and $t \in [0,T]$ that
 \begin{align} \label{eq: Khasminkii C2 conclusion}
 \bar \EE \Big [ \Big |\hat \yY^\varepsilon(t) - \yY^\varepsilon(t)|^2 \textbf{1}_{\{ T < \tilde \tau^\varepsilon_{\rR(\varepsilon)} \}}] \lesssim_\varepsilon \frac{C_2(\varepsilon)}{\Delta(\varepsilon)} e^{- \frac{\Delta(\varepsilon)}{2\varepsilon}+1}  \ra 0 \quad \text{ as } \varepsilon \ra 0,
 \end{align}
 due to the choice of $\Delta (\varepsilon)$ fixed in (\ref{eq: parametrization- the Delta}).

\begin{flushright}
\qed
\end{flushright}

 \paragraph*{Acknowledgments.} The authors acknowledge and thank the financial support from the FAPESP grant number 2018/06531-1 at the University of Campinas (UNICAMP), SP-Brazil.

\end{document}